\documentclass[a4paper,11pt]{amsart}

\usepackage{amsfonts,amssymb,mathtools,amsmath, array, amsthm}
\usepackage{hyperref}
\usepackage{tikz}
\usepackage{textcomp, color}
\usepackage[capitalise]{cleveref}
\usepackage{orcidlink}

\theoremstyle{plain}

\newtheorem*{thmA}{Theorem A}
\newtheorem*{thmB}{Theorem B}
\newtheorem*{thmC}{Theorem C}
\newtheorem{thm}{Theorem}[section]
\newtheorem{lem}[thm]{Lemma}
\newtheorem{pro}[thm]{Proposition}

\newtheorem{rmk}[thm]{Remark}

\theoremstyle{definition}

\newtheorem{dfn}[thm]{Definition}

\newcommand{\F}{\mathbb{F}}
\newcommand{\Z}{\mathbb{Z}}

\newcommand{\N}{\mathbb{N}}

\newcommand{\C}{\mathbb{C}}
\renewcommand{\P}{\mathbb{P}}
\newcommand{\U}{\mathbb{U}}
\newcommand{\T}{\mathbb{T}}

\DeclareMathOperator{\Inn}{Inn}
\DeclareMathOperator{\DInn}{DInn}

\DeclareMathOperator{\Aut}{Aut}
\DeclareMathOperator{\DAut}{DAut}
\DeclareMathOperator{\DDAut}{DDAut}

\DeclareMathOperator{\cl}{cl}

\DeclareMathOperator{\PSL}{PSL}
\DeclareMathOperator{\GL}{GL}
\DeclareMathOperator{\Au}{A_\U}
\DeclareMathOperator{\Out}{Out}
\DeclareMathOperator{\Stab}{Stab}
\DeclareMathOperator{\id}{id}

\begin{document}

\title[The number of Beauville surfaces with Beauville $p$-group]{The Number of Isomorphism Classes of Beauville Surfaces with Beauville $p$-Group}

\author[\c{S}.\ G\"ul]{\c{S}\"ukran G\"ul \orcidlink{0000-0003-4792-7084} }
\address{Department of Mathematics\\ Middle East Technical University
\\
06800 Ankara, Turkey.}
\email{gsukran@metu.edu.tr}

\keywords{Beauville surfaces, Beauville groups; Metacyclic $p$-groups; $p$-groups of class $2$.  \vspace{3pt} \\
{\itshape Mathematics Subject Classification:} 20D15, 14J29 }

\begin{abstract}
A Beauville surface is a rigid complex surface of general type, isogenous
to a higher product by the free action of a finite group $G$, called a Beauville group.
In \cite{GT}, Gonz\'alez-Diez and Torres-Teigell find the number of isomorphism classes of Beauville surfaces for which the group $G$ is $\PSL(2,p)$ with particular types of `Beauville structures'. On the other hand, in \cite{GJT}, Gonz\'alez-Diez, Jones and Torres-Teigell give an explicit formula for this number when the group  $G$ is abelian.
To the best of the author's knowledge, in the literature, the exact number of isomorphism classes of Beauville surfaces is given only for $\PSL(2,p)$ and for abelian groups.
In this paper, we extend the result for Beauville surfaces with abelian $p$-group to Beauville surfaces for which the Beauville group is either a non-abelian metacyclic $p$-group or a $p$-group of nilpotency class $2$.
\end{abstract}	
	
\maketitle

\section{Introduction}

A \emph{Beauville surface\/} of unmixed type is a  compact complex surface isomorphic to 
$(C_1\times C_2)/G$, where $C_1$ and $C_2$ are algebraic curves of genus at least $2$ and $G$ is a finite group acting freely on $C_1\times C_2$ and faithfully on the factors $C_i$ such that $C_i/G\cong \P_1(\C)$ and the covering map $C_i\rightarrow C_i/G$ is ramified over three points for $i=1,2$.
Then the group $G$ is said to be a \emph{Beauville group\/}.

The condition for a finite group $G$ to be a Beauville group can be formulated in purely group-theoretical terms (\cite{BCG, BCG2}).

\begin{dfn}
For a couple of elements $x,y \in G$, we define
\[
\Sigma(x,y)
=
\bigcup_{g\in G} \,
\Big( \langle x \rangle^g \cup \langle y \rangle^g \cup \langle xy \rangle^g \Big),
\]
that is, the union of all subgroups of $G$ which are conjugate to $\langle x \rangle$, to 
$\langle y \rangle$ or to $\langle xy \rangle$. Then $G$ is a Beauville group if and only if the following conditions hold:
\begin{enumerate}
	\item $G$ is a $2$-generator group.
	\item There exists a pair of generating sets $\{x_1,y_1\}$ and $\{x_2,y_2\}$ of $G$ such that 
	\[\Sigma(x_1,y_1) \cap \Sigma(x_2,y_2)=1.\]
\end{enumerate}

If we let $T_1=(x_1, y_1, y_1^{-1}x_1^{-1})$ and $T_2=(x_2, y_2, y_2^{-1}x_2^{-1})$, then the pair of triples $(T_1, T_2)$ is said to be  a \textit{Beauville structure} for $G$.

We write $\U(G)$ for the set of all Beauville structures for a Beauville group $G$.

Futhermore, for $i=1,2$, we say that $T_i$ is of \textit{type} $\tau_i=(m_{i,1},\ m_{i,2},\ m_{i,3})$ if the orders of elements in $T_i$ are $(m_{i,1},\ m_{i,2}, \ m_{i,3})$, respectively.

\end{dfn} 

Beauville surfaces have received a great deal of attention since they were introduced by Catanese in \cite{cat} following a construction of Beauville in \cite{bea}. In \cite{BCG,BCG2, BCG3}, Bauer, Catanese and Grunewald described the basic properties of Beauville surfaces and addressed the most natural questions about them.
Later, they have been studied by Fuertes,  Gonz\'alez-Diez and Jaikin in \cite{FGJ}, and more recently by Fairbairn, Garion, Guralnick, Jones, Larsen, Lubotzky, Magaard, Malle, Parker and Penegini in \cite{FMP,GLL,GP,GM,jon}.

On the other hand, there is a particular interest in determining which finite groups can be used in construction of Beauville surfaces; see, for example, the survey papers \cite{ bos, fai1, fai2,jon2}.
Catanese \cite{cat} showed that a finite abelian group is a Beauville group if and only if it is isomorphic to
$C_n\times C_n$, with $n>1$ and $\gcd(n,6)=1$. A remarkable result, proved independently by Guralnick and Malle \cite{GM} and by Fairbairn, Magaard and Parker \cite{FMP}, is that every non-abelian finite simple group other than $A_5$ is a Beauville group.

After abelian groups the most natural class of finite groups to consider are nilpotent groups. The study of nilpotent Beauville groups is reduced to that of Beauville $p$-groups. In \cite{BBF}, it was shown that there are non-abelian Beauville $p$-groups of order $p^n$ for every  $p\geq 5$ and every $n\geq 3$.
In \cite{SV}, Stix and Vdovina constructed infinite series of Beauville $p$-groups, for every prime $p$, by considering quotients of ordinary triangle groups.  Another infinite family of Beauville $p$-groups, for $p$ an odd prime, was given by G\"ul and Uria-Albizuri \cite{GUA}. Later on, Di Domenico, G\"ul and Thillaisundaram \cite{DDGT} provided more infinite families of Beauville $p$-groups, for every prime $p$. It is also worth mentioning that Fern\'andez-Alcober, G\"ul and Vannacci \cite{FGV} showed that the numbers of Beauville and non-Beauville $p$-groups of fixed order behave asymptotically the same as the number of $2$-generator groups of the same order.

In the specific case of $p$-groups, Fern\'andez-Alcober and G\"ul \cite{FG} extended Catanese's criterion from abelian groups to a much wider family of $p$-groups with a good behaviour with respect to powers. Very recently, in \cite{FGGS}, Fern\'andez-Alcober, Gavioli, G\"ul and Scoppola have studied the existence of Beauville structures for $p$-groups of maximal class.

Note that different Beauville structures for the same group can give rise to non-isomorphic Beauville surfaces.
In \cite{BCG}, Bauer, Catanese and Grunewald gave a lower bound for the asymptotic behavior of the number of isomorphism classes of Beauville surfaces with abelian group $C_n\times C_n$ when $n\rightarrow \infty$. Later, Garion and Penegini \cite{GP} obtained that this number is between $N_n/72$ and $N_n/6$, where $n=p_1^{k_1}\cdots p_t^{k_t}$, for distinct primes $p_i$, and
\[
N_n=\prod_{i=1}^{t}p_i^{4k_i-4}(p_i-1)(p_i-2)(p_i-3)(p_i-4).
\]
Also in \cite{GP}, asymptotic behavior of the number of non-isomorphic Beauville surfaces with the alternating group $A_n$, the symmetric group $S_n$ or the projective special linear group $\PSL(2,p)$ with particular types $(\tau_1, \tau_2)$ was studied.

On the other hand, in \cite{GT}, Gonz\'alez-Diez and Torres-Teigell showed that if $(\tau_1, \tau_2)=\big((2,3,n), (p,p,p) \big)$, for any prime $p\geq 13$ and natural number $n$ diving $(p-1)/2$ or $(p+1)/2$, then the number of isomorphism classes of Beauville surfaces with group $G=\PSL(2,p)$ with types $(\tau_1, \tau_2)$ is exactly equal to $\phi(n)$, where $\phi(n)$ stands for the Euler's Phi function. Later in \cite{GJT}, Gonz\'alez-Diez, Jones and Torres-Teigell gave an explicit formula for the number of isomorphism classes of Beauville surfaces with abelian group $C_n\times C_n$. To the best of the author's knowledge, this and the number given for $\PSL(2,p)$ in \cite{GT} are the only exact number of isomorphism classes of Beauville surfaces in the literature.

Note that the geometric problem of computing the number of isomorphism classes of Beauville surfaces with group $G$ can be reduced into the algebraic one of counting orbits of an action of the group $\Au(G)$, which acts on the set of all
Beauville structures for $G$ by a combination of permutation of coordinates, conjugation, applying diagonally automorphisms of $G$ and applying an operation that exchanges the triples in a Beauville structure. Corollary 2 in \cite{GT} gives that the Beauville surfaces corresponding to two distinct Beauville structures for a group $G$ are isomorphic if and only if those Beauville structures are in the same $\Au(G)$-orbit. In this paper, we will follow this pure group theoretical method. At this point, we would like to emphasize that this is a highly
technical task. Because  first of all, one has to determine all Beauville structures for Beauville group $G$, which can be too complicated for arbitrary Beauville group. Also in order to define the group $\Au(G)$, one needs to determine the automorphism group of $G$, which is also a difficult task. Thus, in this paper we restrict
ourselves to some classes of $p$-groups.

The goal of this paper is to extend the result given in \cite[Corollary 4]{GJT} for the number of isomorphism classes of Beauville surfaces with group $C_{p^e}\times C_{p^e}$ to that for which the Beauville group is either a non-abelian metacyclic $p$-group or a $p$-group of nilpotency class $2$. 

The main results are as follows.

\begin{thmA}
\label{thmA}
	Let $p\geq 5$ be a prime and let $G$ be a Beauville $p$-group given by the following presentation:
	\[
	G=\langle a,b \mid a^{p^e}=b^{p^e}=[b,a]^{p^j}=[b,a,b]=[b,a,a]=1 \rangle
	\]
	where $0<j\leq e$.
	Then the number of isomorphism classes of Beauville surfaces with $G$ is
	\[
	\frac{1}{72}\bigg(p^{4e-4}
	(p-1)(p-2)(p-3)(p-4)
	+ \
	p^{2e-2}  
	\Big(4(p-1)(p-2)+6(p-3)(p-5)\Big)  
	\bigg),
	\]
	if  $p\equiv -1 \pmod{3}$, and
	\[
	\frac{1}{72}
	\bigg( p^{4e-4}
	(p-1)(p-2)(p-3)(p-4)
	+ \
	p^{2e-2}
	\Big(4(p-1)(p-4)+6(p-3)(p-5)\Big)
	+
	24
	\bigg),
	\]
	if $p\equiv 1 \pmod{3}$.
\end{thmA}

\begin{thmB}
\label{thmB}
	Let $p\geq 5$ be a prime and let $G$ be a Beauville $p$-group given by the following presentation:
		\begin{equation*}
			G=\langle a,b \mid a^{p^e}=[b,a]^{p^j}=[b,a,a]=[b,a,b]=1, b^{p^i}=[b,a]^{p^k} \rangle,
		\end{equation*}
		where $0<k<j\leq i\leq e$ and $e= i+j-k$.
		Then the number of isomorphism classes of Beauville surfaces with $G$ is
		\[
		\frac{1}{72}
		\Big(
		p^{4e-6} (p+1)(p-1)^2(p-2)(p-3)(p-4)
		+
		6p^{2e-j-3}(p-1)(p-3)(p-5)
		\Big).
		\]		
\end{thmB}

\begin{thmC}
	\label{thmC}
		Let $p\geq 5$ be a prime and let $G$ be a metacyclic Beauville $p$-group given by the following presentation:
		\begin{equation*}
			G=\langle a,b \mid a^{p^e}=b^{p^e}=1, [a,b]=a^{p^i} \rangle,
		\end{equation*}
		where  $1 \leq i \leq e-1$.
		Then the number of isomorphism classes of Beauville surfaces with $G$ is
		\[
		\frac{1}{72}
		\Big(
		p^{4e-6} (p+1)(p-1)^2(p-2)(p-3)(p-4)
		+
		6p^{2e-3}(p-1)(p-3)(p-5)
		\Big).
		\]	
\end{thmC}

\textit{Organization:} In \cref{characterization}, we state the characterization of non-abelian metacyclic Beauville $p$-groups and Beauville $p$-groups of nilpotency class $2$, and we deal with their automorphism groups. In \cref{all B.S.}, we determine all Beauville structures for the groups given in \cref{characterization}.
In \cref{the group}, we construct the group $\Au(G)$ and calculate its order when $G$ is one of the groups stated in main theorems.
And in \cref{final results}, we prove our main theorems.

\vspace{10pt}

\noindent

\textit{Notation:\/}
We use standard notation in group theory. If $G$ is a group, then we
denote by $\cl(G)$ the number of distinct conjugacy classes of $G$. If $G$ is a finite $p$-group and $i\geq 1$, then  we write $G^{p^i}$ for the subgroup generated by all powers $g^{p^i}$, as $g$ runs over $G$ and $\Omega_i(G)$ for the subgroup generated by the elements of $G$ of order at most $p^i$.

%

\section{Non-abelian metacyclic Beauville $p$-groups and Beauville $p$-groups of nilpotency class $2$}
\label{characterization}

The results of this section are mainly from \cite{AMM, gul, men}.
Let us first state the classification of non-abelian metacyclic Beauville $p$-groups.

\begin{thm}{\cite[Theorem 2.1]{gul}}
	\label{presentation-metacyclic}
	Let $G$ be a non-abelian metacyclic Beauville $p$-group of exponent $p^e$.
	Then $p\geq 5$ and $G$ has the following presentation:
	\begin{equation}
		\label{powerful beauville}
		G=\langle a,b \mid a^{p^e}=b^{p^e}=1, [a,b]=a^{p^i} \rangle,
	\end{equation}
	where  $1 \leq i \leq e-1$.
\end{thm}

\begin{rmk}{\cite[Remark 2.2]{gul}}
Non-abelian metacyclic Beauville $p$-groups are exactly the same as non-abelian powerful Beauville $p$-groups.  
\end{rmk}

\begin{lem}{\cite[p. 1868]{JNS}}
\label{center-powerful}
Let $p$ be an odd prime and let $G$ be a metacyclic $p$-group with the following presentation:
\begin{equation*}
	G=\langle a,b \mid a^{p^e}=b^{p^e}=1, [a,b]=a^{p^i} \rangle,
\end{equation*}
where  $1 \leq i \leq e-1$.
Then $Z(G)=\langle  a^{p^{e-i}}, b^{p^{e-i}} \rangle$.
\end{lem}

%

We next give the automorphism group of non-abelian metacyclic Beauville $p$-groups. To this purpose, we will state the result in \cite{gul}, which refers to the paper \cite[p.82]{men}.

\begin{pro}{\cite[Proposition 2.8]{gul}}
	\label{automorphism-powerful}
	Let $p$ be an odd prime and let $G$ be a metacyclic $p$-group with the following presentation:
	\begin{equation*}
		G=\langle a,b \mid a^{p^e}=b^{p^e}=1, [a,b]=a^{p^i} \rangle,
	\end{equation*}
	where  $1 \leq i \leq e-1$.
	Then a map $\theta$ defined on the generators $a$ and $b$ extends to an automorphism of $G$ if and only if
	\[
	\theta(a)= b^{mp^{e-i}}a^n \ \ \ \
	\text{and} \ \ \ \
	\theta(b)=b^{1+rp^{e-i}}a^{s},
	\]
	where $1\leq m, r\leq p^{i}$ and $1\leq n, s\leq p^e$ with $p\nmid n$. 
	Consequently,we have
	\[|\Aut(G)|=(p-1)p^{2e+2i-1}.\]
\end{pro}

It is easy to observe the following lemma.

\begin{lem}
	\label{inner-powerful}
	Let $p$ be an odd prime and let $G$ be a metacyclic $p$-group with the following presentation:
	\begin{equation*}
		G=\langle a,b \mid a^{p^e}=b^{p^e}=1, [a,b]=a^{p^i} \rangle,
	\end{equation*}
	where  $1 \leq i \leq e-1$.
	Then a map $\theta$ defined on the generators $a$ and $b$ extends to an inner automorphism of $G$ if and only if
	\[
	\theta(a)= a^{1+\mu p^i}\ \ \ \
	\text{and} \ \ \ \
	\theta(b)= ba^{\eta p^i},
	\]
	where $1\leq \mu, \eta \leq p^{e-i}$.
\end{lem}

The following lemma is crucial for the proof of Theorem C.

\begin{thm}
	\label{powerful-elements of Out(G)}
	Let $p$ be an odd prime and let $G$ be a metacyclic $p$-group with the following presentation:
	\begin{equation*}
		G=\langle a,b \mid a^{p^e}=b^{p^e}=1, [a,b]=a^{p^i} \rangle,
	\end{equation*}
	where  $1 \leq i \leq e-1$.
	Then there are $p^{2i}$ involutions in $\Out(G)$.
\end{thm}

\begin{proof}
	Let $\Aut_{\Phi(G)}(G)$ denote the set of automorphisms which induce identity automorphism on $G/\Phi(G)$.
	Let $\alpha$ be a non-inner automorphism such that $\alpha^2 \in  \Inn(G)$. Clearly, $\Inn(G) \leq \Aut_{\Phi(G)}(G) $.
	If $\alpha \in \Aut_{\Phi(G)}(G)$, then since $\Aut_{\Phi(G)}(G)$ is a $p$-group and $p$ is odd, we have $\langle \alpha \rangle= \langle \alpha^2 \rangle \leq \Inn(G)$, which is a contradiction. 
	Thus $\alpha \in \Aut(G)\smallsetminus \Aut_{\Phi(G)}(G)$, that is, $\alpha$ has order $2$ modulo $\Aut_{\Phi(G)}(G)$. 
	
	Note that since $G/\Phi(G)\cong C_p\times C_p$ and each $\theta \in \Aut(G)$ preserves  $\Phi(G)$,  then $\theta$ induces an automorphism $\overline{\theta} \in \Aut(C_p\times C_p)\cong \GL(2, p)$, and the map
	\begin{align*} 
		\rho \colon \Aut(G) & \longrightarrow \Aut(C_p\times C_p)\\ 
		\theta & \longmapsto \overline{\theta}
	\end{align*}
	is a homomorphism. Now if we call $Q= \Aut(G)/ \Aut_{\Phi(G)}(G)$, then we have
 $Q=\rho(\Aut(G))$, and hence $|Q|=p(p-1)$. 
	Indeed, one can easily check that
	\[
	Q\cong A=\left\{ \left(\begin{array}{cc} 
		r & 0\\ 
		\lambda & 1 
	\end{array}\right)
	\in \GL(2,p) \mid r\in \F_p^{*},\  \lambda \in \F_p \right \} \leq \GL(2,p).
	\]
	Choose $r$ to be of order $p-1$. 
	Then $\langle \left(\begin{array}{cc} r & 0\\ \lambda & 1 \end{array}\right)   \rangle$ is a cyclic subgroup of order $p-1$ of $A$. If we write $p-1=2^ik$ for some odd $k \in \N$, then 
	\[
	M= \langle
	{\left(\begin{array}{cc}
			r & 0\\
			\lambda & 1 
		\end{array} \right)}^k \rangle
	\]
	is a Sylow $2$-subgroup of $A$. 
	Note that the number of Sylow $2$-subgroups of $A$ is $|A: N_A(M)|=p$.
	
	We are looking for involutions in $A$. 
	Since Sylow $2$-subgroups of $A$ are cyclic, this implies that there are at most $p$ elements of order $2$ in $A$. 
	Indeed, there are exactly $p$ elements of order $2$ in $A$, which are
	$\left(\begin{array}{cc}
		-1 & 0\\
		\lambda & 1 
	\end{array} \right)$
	for any $\lambda \in \F_p$. 
	Thus, $o(\alpha  \Aut_{\Phi(G)}(G))=2$ in $Q$ implies that 
	\[
	\alpha  \Aut_{\Phi(G)}(G)=\alpha_{\lambda}  \Aut_{\Phi(G)}(G)
	\]
	for some $\lambda \in \F_p$, where $\alpha_{\lambda}$ is represented by the matrix $\left(\begin{array}{cc}
		-1 & 0\\
		\lambda & 1 
	\end{array} \right)$.
	Then $\alpha \in \langle \alpha_{\lambda},  \Aut_{\Phi(G)}(G)\rangle= \langle \alpha_{\lambda} \rangle \ltimes \Aut_{\Phi(G)}(G)$.
	We call
\[
	H= \langle \overline{\alpha_{\lambda}} \rangle \ltimes \overline{\Aut_{\Phi(G)}(G)}\]
	 in $\Aut(G)/ \Inn(G)$. 
	
	Our aim is to find the number of involutions in $H$, that is, the number of Sylow $2$-subgroups of $H$.
	Since $\langle \overline{\alpha_{\lambda}} \rangle$ is a Sylow $2$-subgroup of $H$, this number is equal to
	\[
	|H:N_H(\langle \overline{\alpha_{\lambda}} \rangle)|=|H: C_H(\overline{\alpha_{\lambda}})|.
	\]
	
	Before determining $|H: C_H(\overline{\alpha_{\lambda}})|$, let us see the relation between $\Aut(G)$ and $\GL(2, \Z/ p^e\Z)$. 
	Let 
	\begin{align*} 
		\beta \colon a & \longmapsto a^nb^{mp^{e-i}} \\
		b & \longmapsto b^{1+rp^{e-i}}a^s,
	\end{align*}
	where $p \nmid n$ be an automorphism of $G$. Then there is an injection from $\Aut(G)$ to  $\GL(2, \Z/ p^e\Z)$ such that
	
	\begin{align*} 
		\Aut(G) & \longrightarrow \GL(2, \Z/ p^e\Z) \\
		\beta & \longmapsto 
		\left(\begin{array}{cc}
			n & mp^{e-i}\\
			s & 1+rp^{e-i}
		\end{array} \right)
	\end{align*}
	
	Call 
	$R=\left\{ 
\left(\begin{array}{cc}
	n & mp^{e-i}\\
	s & 1+rp^{e-i}
\end{array} \right) \in \GL(2, \Z/ p^e\Z) \mid p \nmid n \right\}$. 
	Then  we have $\Aut(G) \cong R$. 
	Since all inner automorphisms are of the form $a \rightarrow  a^{1+\mu p^i}$ and $b \rightarrow ba^{\eta p^i}$ and $\Inn(G) \trianglelefteq \Aut(G)$, the subgroup
	\[
	S=\left\{ \left(\begin{array}{cc}
		1+\mu p^i & 0\\
		\eta p^i& 1
	\end{array} \right) \in \GL(2, \Z/ p^e\Z) \right\}
	\]
	is normal in $R$.
	
	Take a fixed element
	$A=\left(\begin{array}{cc}
		n & mp^{e-i}\\
		s & 1+rp^{e-i}
	\end{array} \right) \in R$ and an arbitrary element
	$B=\left(\begin{array}{cc}
		1+\mu p^i & 0\\
		\eta p^i& 1
	\end{array} \right) \in S $.
	Then,  we have
	\[
	AB= \left(\begin{array}{cc}
		n(1+\mu p^i )&  mp^{e-i}\\
		s+ (s\mu+\eta)p^i& 1+rp^{e-i}
	\end{array} \right),
	\]
	and this, together with $|S|=p^{2(e-i)}$, implies that
	\[
	\left(\begin{array}{cc}
		x&  y\\
		u& v
	\end{array} \right) \in \left(\begin{array}{cc}
	n & mp^{e-i}\\
	s & 1+rp^{e-i}
\end{array} \right) S
	\]
	if and only if
	\begin{align} 
		\label{cosets of S}
		y=mp^{e-i}, \ \ \ v=1+rp^{e-i}, \ \ \ x\equiv n \pmod{p^i}, \ 
		\text{and} \ 
		u\equiv s \pmod{p^i}.
	\end{align}
	
	\vspace{15pt}
	We are now ready to determine $|H: C_H(\overline{\alpha_{\lambda}})|$. 
	Let $\overline{\beta} \in \overline{ \Aut_{\Phi(G)}(G)} \cap C_H(\overline{\alpha_{\lambda}})$. 
	Then $\alpha_{\lambda}\beta= \iota_g \beta\alpha_{\lambda}$ for some $g\in G$. Thus
	\[
	\alpha_{\lambda}(\beta(a))= ({\beta(a)}^{-1})^{g}.
	\]
	This equality, together with the fact that  $\beta \in \Aut_{\Phi(G)}(G)$, implies that
	\begin{align*} 
		\beta \colon a & \longmapsto a^{1+\ell p} \\
		b & \longmapsto b^{1+r p^{e-i}}a^{tp}.
	\end{align*}
	Now the equality $\overline{\alpha_{\lambda}}\overline{\beta}=\overline{\beta}\overline{\alpha_{\lambda}}$ implies that the corresponding matrices commute modulo $S$, and hence we have
	\begin{align*} 
	\left(\begin{array}{cc}
			-1-\ell p & 0\\
			\lambda(1+\ell p)+tp & 1+r p^{e-i}
		\end{array} \right) 
	\equiv 
		\left(\begin{array}{cc}
			-1- \ell p & 0\\
			\lambda(1+ rp^{e-i})-tp & 1+r p^{e-i}
		\end{array} \right)
		\pmod{S}.
	\end{align*}
	Then by (\ref{cosets of S}), we conclude that
	\[
	2t\equiv \lambda(r p^{e-i-1}-\ell)
	\pmod{p^{i-1}}.
	\]
	This means that $t$ is determined by  $r$ and $\ell$.
	Hence there are $p^{2i-1}$ possibilities for $\overline{\beta}$. 
	As a consequence, $|C_H(\overline{\alpha_{\lambda}})|=2p^{2i-1}$, and since $|H|=2p^{4i-2}$, we have
	\[
	|H: C_H(\overline{\alpha_{\lambda}})|=p^{2i-1},
	\]
	which is the number of involutions in $H$.
	Since the result holds for any $\lambda \in \F_p$, the number of involutions in $\Out(G)$ is $p^{2i}$. This completes the proof. 
\end{proof}

We next deal with Beauville $p$-groups of nilpotency class $2$.

\begin{pro}{\cite[Proposition 2.3]{gul}}
	\label{p=2}
	There is no Beauville $2$-group of nilpotency class $2$.
\end{pro}

The following result gives the characterization of all Beauville $p$-groups of nilpotency class $2$.

\begin{thm}{\cite[Theorem 2.6]{gul}}
	Let $G$ be a  Beauville $p$-group of nilpotency class $2$ and of exponent $p^e$.
	Then $p\geq 5$ and  $G$ has the following presentation:
	\begin{equation}
		\label{class 2 beauville}
		G=\langle a,b \mid a^{p^e}=[b,a]^{p^j}=[b,a,a]=[b,a,b]=1, b^{p^i}=[b,a]^{p^k} \rangle,
	\end{equation}
	where $0 \leq k \leq  j \leq i \leq e$ and $e=i+j-k$.
\end{thm}

\begin{rmk}
	\label{class 2-types}
	Note that we have three main types of Beauville $p$-groups of nilpotency class $2$.
	The first type is when $k=0$ in \eqref{class 2 beauville}. 
	In this case we have $[b,a]=b^{p^i}$, and hence the group is a non-abelian metacyclic Beauville $p$-group, which has been already examined.
	We will separately deal with the other two cases: $k=j$  or  $0< k < j$.
\end{rmk}

We start with the case $0< k < j$.

\begin{lem}
	\label{class2-case2-center}
Let $p$ be an odd prime and let $G$ be a $p$-group given by the
following presentation:
\begin{equation}
\label{presentation-class2-case2}
G=\langle a,b \mid a^{p^e}=[b,a]^{p^j}=[b,a,a]=[b,a,b]=1, b^{p^i}=[b,a]^{p^k} \rangle,
\end{equation}
where $0<k<j\leq i\leq e$ and $e= i+j-k$.
Then
\[
Z(G)=G^{p^j}G'.
\]
Consequently, we have $|Z(G)|=p^{e+i-j}$.
\end{lem}

\begin{proof}
	Since $G' \leq Z(G)$, this implies that $G^{p^j}=\langle a^{p^j}, b^{p^j}\rangle$.
	One can easily check that $G^{p^j} \leq Z(G)$.
	Now suppose , by way of contradiction, that $G^{p^j}G'$ is  a proper subgroup of $Z(G)$.
	Then there is an element of order $p$ in $Z(G)/G^{p^j}G' $.
	This implies that there is $g\in Z(G)$ such that
	\[
	g\equiv a^{mp^{j-1}}b^{np^{j-1}} \pmod{G'},
	\]
	where either $p\nmid m$ or $p \nmid n$. 
	If $p \nmid m$ then 
	\[
	1=[g,b]=[a^{mp^{j-1}}, b]=[a,b]^{mp^{j-1}},
	\]
	which is a contradiction.
	The similar argument holds if $p \nmid n$. Thus, we conclude that $Z(G)=G^{p^j}G'$.
	
	Since $|G^{p^j}|=|\langle a^{p^j} \rangle||\langle b^{p^j} \rangle|=p^{2e-2j}$ and $G^{p^j}\cap G'=\langle a^{p^i}\rangle$ is of order $p^{e-i}$, we have that $|Z(G)|=p^{e+i-j}$.
\end{proof}

We next give the automorphism group of the group with the presentation given in \cref{presentation-class2-case2}. To this purpose, we will state the result in \cite{gul}, which refers to the paper \cite[p.85]{men}.

\begin{pro} {\cite[Proposition 2.9]{gul}}
	\label{automorphism-class 2-case 2}
	Let $p$ be an odd prime and let $G$ be a $p$-group given by the
	following presentation:
	\[
	G=\langle a,b \mid a^{p^e}=[b,a]^{p^j}=[b,a,a]=[b,a,b]=1, b^{p^i}=[b,a]^{p^k} \rangle,
	\]
	where $0<k<j\leq i\leq e$ and $e= i+j-k$.
	Then a map $\theta$ defined on the generators $a$ and $b$ extends to an automorphism of $G$ if and only if
	\[
	\theta(a)=a^{1+mp^{e-i}}b^nc_a \ \ \ \
	\text{and} \ \ \ \
	\theta(b)= a^{rp^{e-i}}b^sc_b,
	\]
	where  $c_a,c_b \in G'$ and $1\leq m,n,r,s\leq p^i$ with $p \nmid s$.
	Consequently, we have
	\[
	|\Aut(G)|=(p-1)p^{4i+2j-1}.\]
\end{pro}

Now one can easily observe the following lemma.

\begin{lem}	
	\label{class2-case2-inner}
	Let $p$ be an odd prime and let $G$ be a $p$-group given by the
	following presentation:
	\[
	G=\langle a,b \mid a^{p^e}=[b,a]^{p^j}=[b,a,a]=[b,a,b]=1, b^{p^i}=[b,a]^{p^k} \rangle,
	\]
	where $0<k<j\leq i\leq e$ and $e= i+j-k$.
	Then a map $\theta$ defined on the generators $a$ and $b$ extends to an inner automorphism of $G$ if and only if
	\[
	\theta(a)= ac_a\ \ \ \
	\text{and} \ \ \ \
	\theta(b)= bc_b,
	\]
	for any $c_a, c_b \in G'$. 
\end{lem}

We next continue with another lemma which is crucial for the proof of Theorem B.

\begin{thm}
	\label{class2-case2-elements of Out(G)}
	Let $p$ be an odd prime and let $G$ be a $p$-group given by the
	following presentation:
	\[
	G=\langle a,b \mid a^{p^e}=[b,a]^{p^j}=[b,a,a]=[b,a,b]=1, b^{p^i}=[b,a]^{p^k} \rangle,
	\]
	where $0<k<j\leq i\leq e$ and $e= i+j-k$.
	Then there are  $p^{2i-j}$ involutions in $\Out(G)$.
\end{thm}

\begin{proof}
	We will follow the same notations as in \cref{powerful-elements of Out(G)}.
	This leads us to the same conclusion, that is, we have to determine $|H: C_H(\overline{\alpha_{\lambda}})|$. 
	
	Let $\overline{\beta} \in \overline{ \Aut_{\Phi(G)}(G)} \cap C_H(\overline{\alpha_{\lambda}})$. 
	Then $\alpha_{\lambda}\beta= \iota_g \beta\alpha_{\lambda}$ for some $g\in G$. 
	Thus
	\[
	\alpha_{\lambda}(\beta(b))= ({\beta(b)}^{-1})^{g}.
	\]
	This equality, together with the fact that  $\beta \in \Aut_{\Phi(G)}(G)$, implies that
	\begin{align*} 
		\beta \colon a & \longmapsto a^{1+m p^{e-i}}b^{pt}[b,a]^u \\
		b & \longmapsto b^{1+kp}[b,a]^v.
	\end{align*}

	Now, as $\alpha_{\lambda}\beta(b)= b^{-1-kp}[b,a]^{-v}$, and 
	\[
	({\beta(b)}^{-1})^{g}= (b^{-1-kp})^g[b,a]^{-v},
	\]
	we get $g\in \langle b \rangle Z(G)$.
	
	We next consider the equality $\alpha_{\lambda}\beta (a)=\iota_g \beta\alpha_{\lambda}(a)$ where
	\begin{align*} 
		\alpha_{\lambda}\beta (a)=a^{1+m p^{e-i}}b^{\lambda(1+m p^{e-i})-tp}[b,a]^{\lambda mp^{e-i}+\lambda\binom{mp^{e-i}}{2}-u},
	\end{align*}
	and 
	\[
	\iota_g \beta\alpha_{\lambda}(a)=a^{1+m p^{e-i}}b^{\lambda(1+kp)+tp}
	[a,g]^{1+m p^{e-i}}[b,a]^{u+\lambda v}.
	\]
	Then we get
	\[
	2t \equiv \lambda(m p^{e-i-1}-k) \pmod{p^{i-1}}. 
	\]
	Thus $t$ is determined by $m$ and $k$. 
	Also observe that
	\[
	[a,b]^{\lambda v+2u}=[a,b]^{s},
	\]
	for some $s$. 
	Hence
	\[
	2u\equiv s-\lambda v \pmod{p^j}.
	\]
	This means that, once $v$ is chosen, $u$ is uniquely determined.
	
	As a result,  there are $p^{2i+j-1}$ possibilities for $\overline{\beta}$. This implies that $|C_H(\overline{\alpha_{\lambda}})|=2p^{2i+j-1}$, and since $|H|=2p^{4i-2}$, we have
	\[
	|H: C_H(\overline{\alpha_{\lambda}})|=p^{2i-j-1}
	\]
	which is the number of involutions in $H$.
	Since the result holds for any $\lambda \in \F_p$, the number of involutions in $\Out(G)$ is $p^{2i-j}$. This completes the proof. 
\end{proof}

\vspace{15pt}
In the last part, we will deal with the case $k=j$ in  \eqref{class 2 beauville}.

As in the previous case, we have the following easy lemma where the proof is left to the reader.

\begin{lem}
	\label{class2-case3-center}
	Let $p$ be an odd prime and let $G$ be a $p$-group given by the
	following presentation:
	\begin{equation}
		\label{presentation-class2-case3}
		G=\langle a,b \mid a^{p^e}=b^{p^e}=[b,a]^{p^j}=[b,a,b]=[b,a,a]=1 \rangle,
	\end{equation}
	where $0<j\leq e$.
	Then
	\[
	Z(G)=G^{p^j}G'.
	\]
	Consequently, we have $|Z(G)|=p^{2e-j}$.
	
\end{lem}

We close this section by giving the automorphism group of the group with the presentation given in \cref{presentation-class2-case3}.

\begin{pro} {\cite[p.84]{men}}
	\label{automorphism-class 2-case 3}
	Let $p$ be an odd prime and let $G$ be a $p$-group given by the
	following presentation:
	\begin{equation*}
		G=\langle a,b \mid a^{p^e}=b^{p^e}=[b,a]^{p^j}=[b,a,b]=[b,a,a]=1 \rangle,
	\end{equation*}
	where $0<j\leq e$.
	Then a map $\theta$ defined on the generators $a$ and $b$ extends to an automorphism of $G$ if and only if
	\[
	\theta(a)=a^rb^mc_a \ \ \ \
	\text{and} \ \ \ \
	\theta(b)= a^sb^nc_b,
	\]
	where
	$ \left(\begin{array}{cc}
		r & m\\
		s & n \end{array}\right) \in \GL(2,p)$, and $c_a, c_b \in \Phi(G)$.
	Consequently, we have $|\Aut(G)|=(p+1)(p-1)^2p^{4e+2j-3}$.
\end{pro}


\section{Determining All Beauville Structures for a Beauville Group}
\label{all B.S.}
A Beauville group can have many Beauville structures. However, there is no general method for determining all Beauville structures for a Beauville group.

Let us first focus on the abelian Beauville $p$-group $C_p\times C_p$ for a prime $p\geq 5$.

\begin{lem}
	\label{the number of B.S-abelian}
	Let $p\geq 5$ be a prime and let $G\cong C_p\times C_p$. Then the number of all Beauville structures for $G$ is
	\[
	|\U(G)|=(p+1)p(p-1)^3(p-2)(p-3)(p-4).
	\]
\end{lem}

\begin{proof}
It is clear that a pair of generating triples $t=(T_1, T_2)$ is a Beauville structures for $G$ if and only if all of the elements in $T_1\cup T_2$ fall into distinct maximal subgroups of $G$.
Thus, if we write
\[
t=
\Big( (x,y, y^{-1}x^{-1}), (x^ky^{\ell}, x^my^n, x^{-k-m}y^{-\ell-n})\Big),
\]
then $t$ is a Beauville structure for $G$ if and only if
\begin{equation}
	\label{restrictions on powers}
	 k,\ \ell,\ m,\ n,\ k+m,\ \ell+n,\ k-\ell,\ m-n \ \text{and }\  k+m-\ell-n \in U(\Z_p), 
\end{equation}
and $ \left(\begin{array}{cc} k &\ell\\ m & n \end{array}\right) \in \GL(2,p)$.
Once the first triple  has been fixed, then by Lemma 3.21 in \cite{GP}, there are
\[
(p-1)(p-2)(p-3)(p-4)
\]
such quadruples $(k,\ell,m,n)$.

Thus, it remains to determine the number of ways to choose the first triple $(x,y,y^{-1}x^{-1})$.
Since there are $p+1$ maximal subgroups, a maximal subgroup can be chosen in $p+1$ ways. Then any of the $p-1$ non-identity element in the chosen maximal subgroup can be $x$. Then to determine $y$, we can choose one of the remaining $p$ maximal subgroups and any of the $p-1$ non-identity element in the chosen subgroup. This completes the proof.
\end{proof}

The following result gives a sufficient condition for a wide class of non-abelian Beauville $p$-groups $G$ so that every Beauville structure for $G$ is inherited by $G/\Phi(G)$.

\begin{lem} {\cite[Lemma 3.3.1]{gul2}}
	\label{inherited}
	Let $G$ be a Beauville $p$-group of exponent $p^e$ such that for every $x,y\in G$
	\begin{equation}
		\label{quasi-homomorphism}
		x^{p^{e-1}} = y^{p^{e-1}}
		\quad
		\text{if and only if}
		\quad
		(xy^{-1})^{p^{e-1}} = 1.
	\end{equation}
	If $|G:\Omega_{e-1}(G)|\leq p^3$, then every Beauville structure of $G$ is inherited by $G/\Phi(G)$.
\end{lem}

\begin{thm}
\label{number of all B.S.}
Let $G$ be a non-abelian metacyclic Beauville $p$-group or a Beauville $p$-groups of class $2$.Then
\[
|\U(G)|=|\Phi(G)|^4(p+1)p(p-1)^3(p-2)(p-3)(p-4).
\]
\end{thm}

\begin{proof}
Assume that $G$ is a non-abelian metacyclic Beauville $p$-group or a Beauville $p$-groups of class $2$. Then \cref{quasi-homomorphism} holds for $G$ (see \cite[page 3]{FG}). Furthermore,  by \cite[Theorem 3.14]{suz} and \cite[Theorems 1 and 4]{fer} we have
\[
|G:\Omega_{e-1}(G)|=|G^{p^{e-1}}|.
\]
If $G$ a non-abelian metacyclic Beauville $p$-group,  we know that $G$ is powerful. Hence we have $G^{p^{e-1}}=\langle a^{p^{e-1}}, b^{p^{e-1}}\rangle$ is of order $p^2$.
If $G$ is of class $2$, then $G/G^p$ is of order at most $p^3$, where $G^p \leq \Omega_{e-1}(G)$.
Thus, by \cref{inherited}, all Beauville structures of $G$ are inherited by $G/\Phi(G)$. This means that all of the six elements in a Beauville structure $(T_1, T_2)$ for $G$ fall into distinct maximal subgroups.

On the other hand, by the proof of Theorem 2.5 in \cite{FG}, we know that
every lift of a Beauville structure for $G/\Phi(G)$ is a Beauville structure for $G$.
This together with \cref{the number of B.S-abelian} yield that
\[
|\U(G)|=|\Phi(G)|^4(p+1)p(p-1)^3(p-2)(p-3)(p-4).
\]
\end{proof}

\begin{rmk}
\label{all B.S.-rmk}
As a result of \cref{number of all B.S.}, if $G$ is a non-abelian metacyclic Beauville $p$-group or a Beauville $p$-groups of class $2$, then any generating triple $\left(x,y,(xy)^{-1}\right)$ of $G$ appears in some Beauville structures for $G$.
\end{rmk}


\section{The Group $\Au(G)$}
\label{the group}

In this section we will introduce the group $\Au(G)$ that acts on the set of all Beauville structures for a Beauville group $G$.

Firstly, assume that $G$ is a $2$-generator finite group.
We define 
\[
\T(G) := \{(x,y,z) \in G\times G \times G \mid  \  G=\langle x ,y \rangle, \ xyz=1\}.
\]
It is clear that the automorphism group $\Aut(G)$ of $G$ acts diagonally on $\T(G)$, that is, for each $\alpha \in \Aut(G)$ and for each $(x,y,z) \in \T(G)$, the action of $\alpha$ on $(x,y,z)$ is $\left(\alpha(x), \alpha(y), \alpha(z)\right)$. 

For $\alpha \in \Aut(G)$, we use the notation $\widetilde{\alpha}$ in $\DAut(G)$ which refers to the diagonal action of $\Aut(G)$.
In fact the map 
\begin{align*} 
	\Aut(G) & \longrightarrow \DAut(G) \\ 
	\alpha & \longmapsto \widetilde{\alpha}
\end{align*}
is injective.

We additionally define the following permutations of $\T(G)$:
\begin{align*}
&\sigma_0(x,y,z)=(x,y,z), \ \ &\sigma_1(x,y,z)=(y,z,x), \ \ \ &\sigma_2(x,y,z)=(z,x,y)\\
&\sigma_3(x,y,z)=(z,y,x^y), \ \  &\sigma_4(x,y,z)=(y,x,z^x), \ \ \ &\sigma_5(x,y,z)=(x,z,y^z)\\
\end{align*}

\vspace{5pt}

If we set $w:=w(X,Y)$ for a word in $X$ and $Y$, we define the following operation $\beta_{w}$ on $\T(G)$:
\[
\beta_{w}(x, y ,z)=\left(x^{w(x, y)}, y^{w(x, y)}, z^{w(x, y)}\right)
\]


\vspace{5pt}

Consider the words $w_1(X, Y)=Y$ and $w_2(X, Y)=X$. Then the action of the composition $\sigma_i\circ \sigma_j$ on a triple $(x,y,z)$ is given in the following table where the composition $\sigma_i \circ \sigma_j$ can be found in the intersection of the $i$th row and the $j$th column.

\begin{table}[h]
	\begin{tabular}{|c|c|c|c|c|c|c|}
		 \hline
		& \boldmath$\sigma_0$ & \boldmath$\sigma_1$ & \boldmath$\sigma_2 $& \boldmath$\sigma_3$ & \boldmath$\sigma_4$ &\boldmath $\sigma_5$ \\
		\hline
		\boldmath$\sigma_0$ & $\sigma_0$ & $\sigma_1$ & $\sigma_2 $& $\sigma_3$ & $\sigma_4$ & $\sigma_5$\\
		\hline
		\boldmath$\sigma_1$ & $\sigma_1$ & $\sigma_2$ &  $\sigma_0$ & $\beta_{w_2}\circ \sigma_4$ &$\beta_{w_2}\circ \sigma_5$  & $\beta_{w_2}\circ \sigma_3$  \\
		\hline
	\boldmath$\sigma_2$ & $\sigma_2$ & $\sigma_0$  & $\sigma_1$  & $\beta_{w_1^{-1}}\circ \sigma_5$ & $\beta_{w_1^{-1}}\circ \sigma_3$ & $\beta_{w_1^{-1}}\circ \sigma_4$  \\
		\hline
	\boldmath$\sigma_3$ & $\sigma_3$  &  $\sigma_5$& $\sigma_4$  & $\beta_{w_1}\circ \sigma_0$  & $\beta_{w_1}\circ \sigma_2$  & $\beta_{w_1}\circ \sigma_1$ \\
		\hline
	\boldmath$\sigma_4$ & $\sigma_4$ & $\sigma_3$ & $\sigma_5$ & $\sigma_1$ & $\sigma_0$  & $\sigma_2$  \\
		\hline
		\boldmath$\sigma_5$ & $\sigma_5$ & $\sigma_4$  & $\sigma_3$  &  $\beta_{w_2^{-1}}\circ \sigma_2$& $\beta_{w_2^{-1}}\circ \sigma_1$ & $\beta_{w_2^{-1}}\circ \sigma_0$ \\
\hline
	\end{tabular}
\end{table}

Let us write
\[
A(G)= \langle \DAut(G), \sigma_1, \dots, \sigma_5\rangle=\langle \DAut(G), \sigma_1, \sigma_4\rangle
\]

\begin{rmk}
	Observe that $\DAut(G)$ commutes with $\sigma_1$ and $\sigma_4$.
	Thus $\DAut(G)\unlhd A(G)$, and so
	every element $\mu \in A(G)$ can be written as a product of an element in $\langle  \sigma_1, \sigma_4\rangle$ and an element in $\DAut(G)$. 
\end{rmk}

\begin{rmk}
	Using the table above, notice that every element of 
	$\langle  \sigma_1, \sigma_4\rangle$ can be written of the form $\beta_{w}\circ \sigma_i$ for some word $w$ and for some $1\leq i\leq 5$.
	Thus, every element $\mu \in A(G)$ can be written as $\mu=\widetilde{\alpha}\circ\beta_{w} \circ\sigma_i$ where $\alpha \in \Aut(G)$.
\end{rmk}

We also define a subgroup
\[
I(G)= \langle \DInn(G), \sigma_1, \dots, \sigma_5\rangle
\]
of $A(G)$, where $\DInn(G)$ refers to the diagonal action of the inner automorphism group $\Inn(G)$ on $\T(G)$.
Notice that we have
\[
I(G)=  \DInn(G)\langle  \sigma_1, \sigma_4 \rangle.
\]
Then as above, every element of $I(G)$ can be written as 

\[\widetilde{\iota}_g\circ\beta_{w}\circ\sigma_i
\]
for some word $w$ and for some $1\leq i\leq 5$, 
where $\iota_g$  stands for the inner automorphism for some fixed $g\in G$.

Recall that $\U(G)$ is the set of all Beauville structures for a Beauville group $G$.
Now if we set
\[
\DDAut(G):=\{(\widetilde{\alpha}, \widetilde{\alpha} ) \in \DAut(G) \times \DAut(G) \mid  \alpha  \in \Aut(G) \},
\]
then we consider the group $B(G)$ generated by the action of $I(G) \times I(G)$ and by the action of $\DDAut(G)$ on $\U(G)$.

Since $\DAut(G)$ commutes with $\sigma_i$ for any $0\leq i\leq 5$ and $\DInn(G)\unlhd \DAut(G)$, it then follows that $I(G) \times I(G)$ is a normal subgroup of $B(G)$.
Hence
\[
B(G)=\big(I(G) \times I(G)\big)\DDAut(G).
\]

We next define the following operation
\[
\tau\big((a,b,c), (d,e,f)\big)=\big(  (d,e,f), (a,b,c) \big)
\]
on $\U(G)$, and we set 
\[
\Au(G)=\langle B(G), \tau \rangle
\]
for the group generated by these permutations of $\U(G)$.

We are now in a position to state when two pairs of Beauville structures for a Beauville group $G$ give rise to isomorphic Beauville surfaces. To this purpose, we refer to \cite[Corollary 2]{GT} and \cite[Proposition 3.3]{BCG}.

\begin{thm}
Let $G$ be a Beauville group and let $t=(T_1, T_2)$, $t'=(T_1', T_2') \in \U(G)$. Then the Beauville surfaces corresponding to $t$ and $t'$ are isomorphic if and only if $t$ and $t'$ are in the same $A_{\U}(G)$-orbit.
\end{thm}

As a result, in order to compute the number of isomorphism classes of Beauville surfaces with group $G$, we need to count the number of orbits of the group $A_{\U}(G)$ on the set $\U(G)$. By using Cauchy-Frobenius Lemma, this number is equal to
\[
\dfrac{1}{|A_{\U}(G)|}\sum_{t\in \U(G)}|\Stab_{A_{\U}(G)}(t)|
\]

The following lemma and proposition will be used in determination of the order of $\Au(G)$.

\begin{lem}
	\label{index of B(G)}
	Let $G$ be a nilpotent Beauville group.
	Then we have that $|\Au(G):B(G)|=2$.
\end{lem}

\begin{proof}
	Suppose, on the contrary, that $\tau \in B(G)$. Then we have
	\[
	\tau= \big(
	\widetilde{\alpha} \circ \widetilde{\iota}_g \circ \beta_{w} \circ \sigma_i, \ \  \widetilde{\alpha} \circ \widetilde{\iota}_h \circ \beta_{w'} \circ \sigma_j
	\big)
	\]
	for some words $w$ and $w'$, for some $\alpha \in \Aut(G)$ and for
	$0\leq i, j \leq 5$.

	Let $(T_1, T_2)$ form a Beauville structure for $G$. Set $T_1=(x, y, z)$. If we set $T_1'=(y, x, z^x)$, then $(T_1', T_2)$ also form a Beauville structure for $G$. Hence
	\[
	\widetilde{\alpha} \circ \widetilde{\iota}_g \circ \beta_{w} \circ \sigma_i(x,y,z)=T_2
	\ \ \
	\text{and}
	\ \ \
	\widetilde{\alpha} \circ \widetilde{\iota}_g \circ \beta_{w} \circ \sigma_i(y, x, z^x)=T_2.
	\]
	This implies that
	\[
	\beta_{w} \circ \sigma_i(x,y,z)= \beta_{w} \circ \sigma_i(y, x, z^x).
	\]
	Since $G$ is nilpotent, we have $G'\leq \Phi(G)$. This, together with the fact that $x, y$ and $z$ fall in three different maximal subgroups, implies that the above equality cannot hold for any $i$ with $0\leq i\leq 5$ and for any word $w$.
	Thus, $\tau \notin B(G)$.
\end{proof}

\begin{pro}
\label{quotient by J}
Let $G$ be a two generator finite group. Then the set
\[
J(G)=\{ \beta_w \mid w \ \ \text{is a word in a set with two elements} \}
\]
is a normal subgroup of $\langle \sigma_1, \sigma_4 \rangle$.

Furthermore, if $G$ is a nilpotent group, then $\langle \sigma_1, \sigma_4 \rangle /J(G)\cong S_3$.
\end{pro}

\begin{proof}
Since $\sigma_3^2 (x, y, z)=(x, y, z)^{y}$ and $\sigma_5^2(x, y,z)=(x, y,z)^{x^{-1}}$, it then follows that
\[
L=\langle \beta_{w} \mid w \ \ \text{is a word in a set with two elements} \rangle.
\]
is a subgroup of $\langle \sigma_1, \sigma_4 \rangle=\langle \sigma_1, \cdots, \sigma_5 \rangle$.

We will show that $L$ coincides with $J(G)$.
To this purpose, take two elements $\beta_{w_1}, \beta_{w_2}$ from the generating set of $L$.
Then
\[
\beta_{w_2}  \circ \beta_{w_1}(x, y, z)= 
(x^{w_1(x, y)}, y^{w_1(x, y)}, z^{w_1(x, y)})^
{w_2(x^{w_1(x, y)}, y^{w_1(x, y)})}.
\]
Since $w_2(x^{w_1(x, y)}, y^{w_1(x, y)})=w_2(x, y)^{w_1(x, y)}$, it then follows that
$\beta_{w_2}\circ\beta_{w_1}=\beta_{w_2w_1} \in J(G)$.
Thus, $L=J(G)$.

We next show that $J(G)$ is normal in $\langle \sigma_1, \sigma_4 \rangle$. Let $\beta_w \in J(G)$. Recall that for any triple $(x ,y ,z)$ we have $z=y^{-1}x^{-1}$.

Now call 
$w'(X, Y)=w(Y, Y^{-1}X^{-1})$. 
Then we have 
\[
\sigma_2\circ\beta_w \circ \sigma_1=\beta_{w'}  \in J(G).
\]
Similarly, if we call $w''(X, Y)=w(Y, X)$, then
\[
\sigma_4\circ\beta_w \circ \sigma_4=\beta_{w''}  \in J(G).
\]
Thus, $J(G)\unlhd \langle \sigma_1, \sigma_4 \rangle$.

Now suppose that $G$ is a finite nilpotent group. Then we have $G'\leq \Phi(G)$. 
Let $(x, y, z)$ be a generating triple of $G$.
Then for every $\beta_{w} \in J(G)$, each component of $\beta_w(x, y, z)$ and the corresponding component of $(x, y, z)$ fall in the same maximal subgroup.
This implies that  $\sigma_1, \dots, \sigma_5$ do not fall  in $J(G)$.

If we use the bar notation in $\langle \sigma_1, \sigma_4 \rangle/J(G)$, we have
\[\langle \sigma_1, \sigma_4 \rangle/J(G)=\langle \overline{\sigma_1},  \overline{\sigma_4} \rangle
\]
and $\overline{\sigma_1}^3= \overline{\sigma_4}^2=\overline{1}$, and 
$ \overline{\sigma_3}^2= \overline{\sigma_5}^2=\overline{1}$. 
Also, since
\[
\sigma_4\circ\sigma_1\circ \sigma_4(x,y,z)=(z,x,y)^x
\]
we have $\overline{\sigma_4}\overline{\sigma_1}\overline{\sigma_4}=\overline{\sigma_2}$. Therefore, $\langle \sigma_1, \sigma_4 \rangle/J(G) \cong S_3$.
\end{proof}

We continue this section by examining the order of $A_{\U}(G)$ when $G$ is either a non-abelian metacyclic $p$-group or a $p$-group of nilpotency class $2$.

\begin{pro}
\label{intersections-metacyclic}
Let $G$ be a non-abelian metacyclic Beauville $p$-group.
Then we have the following:

\begin{enumerate}
	\item 
	$\DInn(G) \cap \langle \sigma_1, \sigma_4 \rangle=1$.
	
	\vspace{5pt}
	
	\item 
	$\big( I(G)\times I(G)\big) \cap \DDAut(G)
	= 
	\big\{ (\widetilde{\iota}_g, \widetilde{\iota}_g) \in \DInn(G)\times \DInn(G) \mid g \in G \big\}.$
\end{enumerate}
\end{pro}

\begin{proof}
Recall that, by \cref{presentation-metacyclic},
\[
G=\langle a,b \mid a^{p^e}=b^{p^e}=1, [a,b]=a^{p^i} \rangle
\]
where $p\geq 5$ and $1\leq i\leq e-1$.
Also we know that, by \cref{all B.S.-rmk}, every geneating triple for $G$ appears in some Beauville structure for $G$.

We will prove (i) and (ii) simultaneously.
Let $\psi \in \DInn(G) \cap \langle \sigma_1, \sigma_4 \rangle$ and 
let $( \phi, \rho) \in \big( I(G)\times I(G)\big) \cap \DDAut(G)$. 
Then we write 
\[\psi=\beta_{w} \circ \sigma_i=\widetilde{\iota}_h
\]
for some $h\in G$ and for some $0\leq i \leq 5$. 
Similarly, we write 
\[
\phi= \widetilde{\iota}_g \circ \beta_{w'} \circ \sigma_j= \widetilde{\alpha}
\]
for some $\alpha \in \Aut(G)$ and  for some $0\leq j \leq 5$. 
In order to simply the proof we may assume that $i=j$ and $w=w'$.
Then we have
$\beta_w\circ \sigma_i=\widetilde{\iota}_{g^{-1}} \circ\widetilde{\alpha}$ in case (ii).

Now we call $\widetilde{\gamma}=\widetilde{\iota}_h$ in (i), and $\widetilde{\gamma}=\widetilde{\iota}_{g^{-1}} \circ\widetilde{\alpha}$ in (ii). 
This implies that $\beta_w\circ \sigma_i=\widetilde{\gamma}$ in both cases.

We will first show that $i=0$. Now consider the generating triple $T_1=(a,b, (ab)^{-1})$. Then in (i), we have 
\[
\beta_w\big(\sigma_i\left(a, b, (ab)^{-1}\right) \big)=\widetilde{\gamma}(a, b, (ab)^{-1}),
\]
where $\gamma$ is an inner automorphism.
Since $G'\leq \Phi(G)$, and $a, b$ and $(ab)^{-1}$ fall into three distinct maximal subgroups of $G$, it then follows that $\sigma_i=\sigma_0$. 

Assume now that we are in case (ii). Observe that $T_2=(b,a,(ba)^{-1})$ is a generating triple for $G$.

Suppose that $i=1$.
Then the equality $\beta_w\circ \sigma_i=\widetilde{\gamma}$ implies that
\[
\left(b, (ab)^{-1}, a\right)^{w(b, (ab)^{-1})}=\widetilde{\gamma}\left(a, b, (ab)^{-1}\right)
\]
and
\[
\left(a, (ba)^{-1}, b \right)^{w(a, (ba)^{-1})}= \widetilde{\gamma}\left(b, a, (ba)^{-1}\right),
\]
and so  $\gamma(a)=b^{w\left(b, (ab)^{-1}\right)}=\left((ba)^{-1}\right)^{w\left(a, (ba)^{-1}\right)}$ which is not possible since $a, b$ and $(ab)^{-1}$ fall into three distinct maximal subgroups and $G'\leq \Phi(G)$.
Similarly, we have $i\neq 2, 3, 5$. 

Next assume that $i=4$.
Now consider the generating triple $T_3=\big(b^{-1}, ab, (a^{-1})^{b} \big)$ for $G$.
Then we have
\[
\left(b, a, (ba)^{-1}\right)^{w(b, a)}=\widetilde{\gamma}\left(a, b, (ab)^{-1}\right)
\]
and
\[
\left(ab, b^{-1}, a^{-1}\right)^{w\left(ab,  b^{-1}\right)}
	= \widetilde{\gamma}\left(b^{-1}, ab, (a^{-1})^{b}\right).
\]
Thus $\gamma(b)=a^{w(b,a)}=\left((ab)^{-1}\right)^{w(ab, b^{-1})}$, which is again not possible. Therefore $i=0$.

Hence in both cases we have $\beta_w=\widetilde{\gamma}$.

Now our claim is that for any generating pair $(x, y)$ of $G$, we have $w(x, y) \in Z(G)$. 
Since $w$ is a word in $X$ and $Y$, we write $w(X, Y)= X^rY^s c(X,Y)$ where $c(X,Y)$ is a product of commutators in $X$ and $Y$ and $r, s \in \Z$.

By using $T_1$ and $T_2$, the equality $\beta_w=\widetilde{\gamma}$ implies that
\[
(a, b, (ab)^{-1})^{w(a, b)}=\widetilde{\gamma}(a, b, (ab)^{-1}) \ \
\text{and}
\ \
(b, a, (ba)^{-1})^{w(b, a)}= \widetilde{\gamma}(b, a, (ba)^{-1}).
\]
Hence
\[
w(a, b) \equiv w(b, a) \pmod{Z(G)},
\]
where 
$w(a, b)=b^s a^r a^{p^it_1}$ and $w(b, a)=b^ra^s a^{p^it_2}$ for some $1\leq t_1, t_2 \leq p^{e-i}$. 
Then 
\[
b^{s-r}\equiv a^{s-r}a^{p^i(t_2-t_1)} \pmod{Z(G)},
\]
where $Z(G)=\langle  a^{p^{e-i}}, b^{p^{e-i}} \rangle$ by \cref{center-powerful}. 
Thus $p^{e-i} \mid s-r$. 
Since $p^{e-i}$th power of all elements of $G$ are central, we may assume that  $r=s$ and $1\leq r \leq p^{e-i}$.

Next consider another generating triple  $T_3=(ab, b, (ab^2)^{-1})$ for $G$. Then  we get
\[
(ab, b, (ab^2)^{-1})^{w(ab, b)}= \widetilde{\gamma}(ab, b, (ab^2)^{-1}).
\]
This implies that $\gamma(b)=b^{w(ab, b)}$ 
and $\gamma(a)=a^{w(ab, b)}$. Consequently, we also have
\[
w(ab, b)\equiv w(a,b)\pmod{Z(G)},
\]
where $w(ab, b)= b^{2r}a^r a^{p^i t_3}$ for some $1\leq t_3 \leq p^{e-i}$. 
Then we get
\[
b^r \equiv a^{p^i(t_1-t_3)} \pmod{Z(G)},
\]
and hence $p^{e-i} \mid r$. Thus, we may assume that $w(X, Y)= c(X, Y)$ is a product of commutators in $X$ and $Y$.

Now we take the generating triple
$T_4=(a^{-1}, b, b^{-1}a)$. 
Then
\[
(a^{-1}, b, b^{-1}a)^{w(a^{-1}, b)}
= \widetilde{\gamma}(a^{-1}, b, b^{-1}a).
\]
Thus $w(a, b) \equiv w(a^{-1}, b) \pmod{Z(G)}$. 

On the other hand, since $G'=\langle a^{p^i}\rangle$ is abelian, this yields that
\[
w(a, b)=\prod_{j=0}^{n-2}\ [[a,b], b, \overset{j}{\cdots}, b]^{\ell_j},
\]
where $\ell_j\in \Z$ and $n=\cl(G)$.
Observe also that
\[
w(a^{-1}, b)=\prod_{j=0}^{n-2}\ [[a,b], b, \overset{j}{\cdots}, b]^{-\ell_j}.
\]
Thus, the congruence  $w(a, b) \equiv w(a^{-1}, b) \pmod{Z(G)}$ implies
that 
\[
\prod_{j=0}^{n-2} \ [[a,b], b, \overset{j}{\cdots}, b]^{2\ell_j}\in Z(G).
\]
Since $p$ is odd, we conclude that $w(a, b) \in Z(G)$, and hence
\[
(a, b, b^{-1}a^{-1})^{w(a, b)}= (a, b, b^{-1}a^{-1})
=\widetilde{\gamma}(a, b, b^{-1}a^{-1}).
\]
Consequently, $\gamma=\id_G$. 

Hence, this means that $\iota_h=\id_G$ in case (i), and  $\alpha=\iota_g$ in case (ii).
This completes the proof.
\end{proof}

\begin{pro}
	\label{intersections-class 2}
	Let $G$ be a Beauville $p$-group of class $2$.
	Then 
	\begin{enumerate}
		\item 
		$\DInn(G) \cap \langle \tau_1, \tau_3 \rangle=1$,
		\item 
		$\big( I(G)\times I(G)\big) \cap \DDAut(G)
		= 
		\big\{ (\widetilde{\iota}_g, \widetilde{\iota}_g) \in \DInn(G)\times \DInn(G) \mid g \in G \big\}.$
	\end{enumerate}
\end{pro}

\begin{proof}
	First of all, note that since $G$ is of class $2$,  we have $G/Z(G)\cong C_{p^m} \times C_{p^m}$ where $\exp G'=p^m$. 
	Recall that by \cref{all B.S.-rmk}, every generating triple for $G$ appears in some Beauville structure for $G$.
	
	We are going to prove (i) and (ii) simultaneously.
	We will use the same arguments in the proof of  \cref{intersections-metacyclic}.
	Then one can show that in both cases $\beta_w=\widetilde{\gamma}$, where $\widetilde{\gamma}=\widetilde{\iota}_h$ in (i), and $\widetilde{\gamma}=\widetilde{\iota}_{g^{-1}} \circ\widetilde{\alpha}$ in (ii).

	Since $G$ is of class $2$, we can assume that $w(x, y)=x^iy^j$ with $1 \leq i, j \leq p^m$. 
	Then by using the generating triples $T_1=(x,y,z)$ and $T_2=\left( y,x, (yx)^{-1}\right)$ we get 
	\[
	x^iy^j\equiv x^jy^i \pmod{Z(G)},
	\]
	that is $x^{i-j}\equiv y^{i-j} \pmod{Z(G)}$.
	Thus we have $i=j$. 
	Finally we consider the generating triples $T_3=(xy, y^{-1}, x^{-1})$ and $T_4=(x^{-1}, xy, y^{-1})$.
	Then we have
	\[
	(xy, y^{-1}, x^{-1})^{x^i}=\widetilde{\gamma}(xy, y^{-1}, x^{-1}) \ \
	\text{and}
	\ \
	(x^{-1}, xy, y^{-1})^{y^i}=\widetilde{\gamma}(x^{-1}, xy, y^{-1}).
	\]
	Thus $\gamma(x)=x=x^{y^i}$ and $\gamma(y)=y=y^{x^i}$.
	Hence we have $x^i, y^i \in Z(G)$. 
	Then we conclude that $w( x, y)=x^iy^j \in Z(G)$ for any generating pair $(x,y)$ of $G$.
	Consequently,  $\gamma=\id_G$.
	This completes the proof. 
\end{proof}

We close this section by giving a formula for the order of $A_{\U}(G)$.

\begin{lem}
	\label{order of the group}
	Let $G$ be a non-abelian metacyclic Beauville $p$-group or a Beauville $p$-groups of class $2$.
	Then
	\[
	|\Au(G)|=72|J(G)|^2|\Aut(G)||\Inn(G)|.
	\]
\end{lem}

\begin{proof}
	We know that $|\Au(G)|=2|B(G)|$ by \cref{index of B(G)}.
	Also by \cref{quotient by J},  we have
	$|\langle \sigma_1, \sigma_4\rangle|=6|J(G)|$.
	Then  \cref{intersections-metacyclic} and \cref{intersections-class 2} imply that
	\begin{align*}
		2|B(G)|&=2\frac{|I(G)|^2 |\Aut(G)|}{|\Inn(G)|}\\
		&=72 |J(G)|^2|\Aut(G)||\Inn(G)|.
	\end{align*}
\end{proof}


\section{Proofs of the Main Results}
\label{final results}

In this section, we will prove the main theorems. To this purpose, we will start with an easy lemma.

\vspace{5pt}

Note that if $G$ is a $2$-generator finite $p$-group and $\alpha \in \Aut(G)$, then the map
\begin{equation} 
	\label{induced automorphism}
	\begin{split}
	f\colon \Aut(G) & \longrightarrow \Aut(C_p\times C_p)\\ 
	\alpha & \longmapsto \alpha'
\end{split}
\end{equation}
is a homomorphism.

\begin{lem}
\label{inherited stabilizer}
Let $G$ be a Beauville $p$-group and let $t$ be a Beauville structure for $G$ which is inherited by $G/\Phi(G)$. If we use the bar notation in $G/\Phi(G)$, and
\[
(\widetilde{\alpha}\circ\beta_w \circ\sigma_i, \
\widetilde{\alpha}\circ\widetilde{\iota}_g\circ\beta_{w'}\circ\sigma_j ) \in \Stab_{\Au(G)}(t)
\]
for some $\alpha \in \Aut(G)$, $g\in G$, then

\[
(\widetilde{\alpha'}\circ\sigma_i, \ \widetilde{\alpha'}\circ\sigma_j ) \in \Stab_{\Au(C_p\times C_p)}(\overline{t})
\] 
where $\alpha'$ is the image of $\alpha$ under the map $f$ defined in \cref{induced automorphism}.
\end{lem}

\begin{proof}
	Let $t=(T_1, T_2)$.
	Since
	$(\widetilde{\alpha}\circ\beta_w \circ\sigma_i, \
	\widetilde{\alpha}\circ\widetilde{\iota}_g\circ\beta_{w'}\circ\sigma_j ) \in \Stab_{\Au(G)}(t)$, 
	it then follows that
	\[
	\beta_w\circ\sigma_i(T_1)=\widetilde{\alpha}^{-1}(T_1) \ \
	\text{and}
	\ \
	\beta_{w'}\circ\sigma_j(T_2)= \widetilde{\iota}_{g^{-1}}\circ\widetilde{\alpha}^{-1}(T_2).
	\]
	Since $\overline{t}$ is a Beauville structure for $\overline{G}\cong C_p \times C_p$, we have
	\[
	\sigma_i(\overline{T}_1)
	=
	\widetilde{\alpha'}^{-1}(\overline{T}_1)\ \
	\text{and}
	\ \
	\sigma_j(\overline{T}_2)
	=
\widetilde{\alpha'}^{-1}	(\overline{T}_2).
	\]
	As a consequence, we have
	$(\widetilde{\alpha'}\circ\sigma_i, \ \widetilde{\alpha'}\circ\sigma_j ) \in \Stab_{\Au(C_p\times C_p)}(\overline{t})$.
\end{proof}

The following theorem is crucial for the proof of Theorem A.

\begin{thm}
	\label{stabilizer in B(G)}
		Let $p\geq 5$ be a prime and let $G$ be a Beauville $p$-group given by the
	following presentation:
	\begin{equation*}
		G=\langle a,b \mid a^{p^e}=b^{p^e}=[b,a]^{p^j}=[b,a,b]=[b,a,a]=1 \rangle,
	\end{equation*}
	where $0<j\leq e$.
	Let 
	\[
	t=
	\Big( \big(x,y, y^{-1}x^{-1}\big), \big(x^ky^{\ell}c, x^my^nd,  (x^my^nd)^{-1}(x^ky^{\ell}c)^{-1}\big)\Big)
	\]
	be a Beauville structure for $G$ where
	$c, d \in G'$ and $1\leq k, \ell, m, n \leq p^e$. 
	Then for some $\alpha \in \Aut(G)$ and for some $h\in G$, we have
	\begin{equation}
		\label{stabilizers in B(G)-case distinction}
		\Stab_{B(G)}(t) =
		\begin{cases}
			S\langle (\widetilde{\alpha}\circ\sigma_2, \widetilde{\alpha}\circ\widetilde{\iota}_{h}\circ\sigma_2) \rangle
			&\text{if \ $\ell+m\equiv 0\pmod{p^e}$ and}
			\\
			&n+\ell-k \equiv 0 \pmod{p^e},
			\\ \\
			S\langle (\widetilde{\alpha}\circ\sigma_2, \widetilde{\alpha}\circ\widetilde{\iota}_{h}\circ\sigma_1)\rangle
			&\text{if \ $k+n \equiv 0\pmod{p^e}$ and}
			\\
			&k+m-\ell \equiv 0 \pmod{p^e} ,
			\\ \\
			S
			&\text{otherwise},
		\end{cases}
	\end{equation}
	where
	$S=
	\big\{ 
	\big(\widetilde{\iota}_{w(x,y)^{-1}}\circ  \beta_w, \ \widetilde{\iota}_{w'(x^ky^{\ell}c,\ x^my^nd)^{-1}}\circ\beta_{w'}  \big)
	\mid
	\beta_w, \beta_{w'} \in J(G)
	\big \}$.
\end{thm}

\begin{proof}
	Let $t=(T_1, T_2)$, and write for simplicity $T_2=(u, v, v^{-1}u^{-1})$. 
	Suppose that
	\[
	(\widetilde{\alpha}\circ\beta_w\circ\sigma_i, \
	\widetilde{\alpha}\circ\widetilde{\iota}_g\circ\beta_{w'}\circ\sigma_j) \in \Stab_{B(G)}(t)
	\]
	for some $\alpha\in \Aut(G)$ and $g\in G$.
	
	Assume first that $\sigma_i=\sigma_0$. 
	Then the equality $\widetilde{\alpha}\circ\beta_w(T_1)=T_1$ implies that $\alpha=\iota_{w(x,y)^{-1}}$.
	Now the condition \[\widetilde{\alpha}\circ\widetilde{\iota}_g\circ\beta_{w'}\circ\sigma_j(T_2)=T_2
	\]
	holds if and only if 
	\[
	\sigma_j=\sigma_0 \ \
	\text{and} \ \
	g \equiv w'(u,v)^{-1}w(x,y) \pmod {Z(G)},
	\]
	and hence
	$
	(\widetilde{\alpha}\circ\beta_w, \
	\widetilde{\alpha}\circ\widetilde{\iota}_g\circ\beta_{w'}) 
	=
	(\widetilde{\iota}_{w(x,y)^{-1}}\circ\beta_w, \
	\widetilde{\iota}_{w'(u,v)^{-1}}\circ\beta_{w'}).
	$
	Thus, we have
	\begin{equation}
		\label{common elements in stablizers}
		S=
		\big\{ 
		\big( \widetilde{\iota}_{w(x,y)^{-1}}\circ\beta_w, \
		\widetilde{\iota}_{w'(u,v)^{-1}}\circ\beta_{w'}  \big)
		\mid
		\beta_w, \beta_{w'} \in J(G)
		\big \} \leq \Stab_{B(G)}(t).
	\end{equation}

	Now suppose that $\sigma_i \in \{ \sigma_3, \sigma_4, \sigma_5\}$. 
	We know that, by the proof of \cref{number of all B.S.}, all Beauville structures of $G$ are inherited by $G/\Phi(G)$. If we use the  bar notation in $G/\Phi(G)$, then
	by \cref{inherited stabilizer}, we have
	\[
	(\widetilde{\alpha'}\circ\sigma_i , \widetilde{\alpha'}\circ\sigma_j ) \in \Stab_{\Au(C_p\times C_p)}(\overline{t}).
	\]
	It then follows that $\alpha'$ fixes one of the elements in $\{ \overline{x}, \overline{y}, (\overline{x}\overline{y})^{-1}\}$.
	Without loss of generality, suppose that $\alpha'(\overline{x})=\overline{x}$ and $\alpha'(\overline{y})=(\overline{x}\overline{y})^{-1}$.
	Then observe that $\sigma_j \in \{\sigma_1, \sigma_2\}$.
	Otherwise $\alpha'$ would fix one element in $\overline{T}_2$, and this would imply that $\alpha'=id_{\overline{G}}$.
	Then we would have $\sigma_i=\sigma_0$, which is not the case.
	
	Now since $\sigma_j \in \{\sigma_1, \sigma_2\}$, then
	$\alpha'(\overline{x}^k\overline{y}^{\ell})= \overline{x}^{k-\ell}\overline{y}^{-\ell}$ is equal to either $\overline{x}^m\overline{y}^{n}$ or $\overline{x}^{-k-m}\overline{y}^{-\ell-m}$.
	However, neither of them are possible, by \eqref{restrictions on powers}.
	Consequently, $\sigma_i, \sigma_j \in \{\sigma_1, \sigma_2 \}$.
	
	Now let $\sigma_i=\sigma_j=\sigma_2$. 
	Then  we have 
		\[
	(\widetilde{\alpha}\circ\beta_w\circ\sigma_2, \
	\widetilde{\alpha}\circ\widetilde{\iota}_g\circ\beta_{w'}\circ\sigma_2) \in \Stab_{B(G)}(t).
	\]
	By taking into account \eqref{common elements in stablizers}, for simplicity we may assume that
	\[(\widetilde{\alpha}\circ\sigma_2, \widetilde{\alpha}\circ\widetilde{\iota}_g\circ\sigma_2) \in \Stab_{B(G)}(t).
	\]
	This implies that $\alpha(x)=y$ and $\alpha(y)=y^{-1}x^{-1}$. 
	Also we have 
	\[
	\alpha(x^ky^{\ell}c)^{\alpha(g)}=x^my^{n}d
	\] 
	and \[\alpha(x^my^{n}d)^{\alpha(g)}=(x^my^{n}d)^{-1}(x^ky^{\ell}c)^{-1}.
	\]
	By the 
	former equality, we have
	\[
	x^{-\ell}y^{k-\ell} [x,y]^{\ell(k-\ell)+ \binom{\ell}{2}}[y^{k-\ell}x^{-\ell},\alpha(g)]c=x^my^{n}d.
	\] 
	Thus  
	\begin{equation}
	\label{exponent relations}
	m\equiv -\ell \pmod{p^e}, \ \ \ 
	\text{and} \ \ \ 
	n +\ell-k \equiv 0 \pmod{p^e},
	\end{equation}
	and 
	\begin{equation}
		\label{c-d relation 1}
		d=[x,y]^{\ell(k-\ell)+ \binom{\ell}{2}}[y^{k-\ell}x^{-\ell},\alpha(g)]c.
	\end{equation}
	By  the equality $\alpha(x^my^{n}d)^{\alpha(g)}=(x^my^{n}d)^{-1}(x^ky^{\ell}c)^{-1}$, together with \cref{exponent relations}, we get
	\[
	y^{-k}x^{\ell-k}[x,y]^{\binom{k-\ell}{2}}
	[y^{-k}x^{\ell-k}, \alpha(g)]d=y^{-k}x^{\ell-k}[x,y]^{-\ell^2}d^{-1}c^{-1}.
	\]
	Thus 
	\begin{equation}
		\label{c-d relation 2}
		d^2[y^{-k}x^{\ell-k}, \alpha(g)]=[x,y]^{-\ell^2-\binom{k-\ell}{2}}c^{-1}.
	\end{equation}
	
Now	by using \eqref{c-d relation 1} and \eqref{c-d relation 2}, we get
	\begin{equation}
		\label{c-d relation 3}
		[y^{k-2\ell}x^{-k-\ell}, \alpha(g)]=[x,y]^{f(k,\ell)}c^{-3},
	\end{equation}
	for some function $f(k,\ell)$ depending on $k$ and $\ell$. 
	Note that since $y^{k-2\ell}x^{-k-\ell} \notin \Phi(G)$ we have
	\[
	G'= \{[y^{k-2\ell}x^{-k-\ell}, w] \mid w\in G  \}.
	\]
	As a consequence, there exists $\alpha(g) \in G$ satisfying \eqref{c-d relation 3}, and it is unique modulo
	\[
	C_G(y^{k-2\ell}x^{-k-\ell})= \langle y^{k-2\ell}x^{-k-\ell} \rangle Z(G).
	\]
	If $\alpha(h)=(y^{k-2\ell}x^{-k-\ell})^r\alpha(g)$ for some $r\in \N$, then by \eqref{c-d relation 1} we have
	\begin{align*}
		d'&= [x,y]^{\ell(k-\ell)+ \binom{\ell}{2}} [y^{k-\ell}x^{-\ell}, \alpha(h)]c \\
		&=d[y^{k-\ell}x^{-\ell}, y^{k-2\ell}x^{-k-\ell}]^r.
	\end{align*}
	If we call $s=y^{k-\ell}x^{-\ell}$ and $q=y^{k-2\ell}x^{-k-\ell}$, then $G=\langle s, q \rangle$, and hence $G'=\langle [s,q]\rangle$. 
	Thus, for any $c \in G'$, there correspond $|G'|$ elements $d' \in G'$, which are the elements in the coset $dG'$, such that 
	$(\widetilde{\alpha}\circ\sigma_2, \widetilde{\alpha}\circ\widetilde{\iota}_{h}\circ\sigma_2) \in \Stab_{B(G)}(t)$. 
	Consequently,  for any Beauville structure $t$ satisfying
	\[
	m\equiv -\ell \pmod{p^e}, \ \ \ 
	\text{and} \ \ \ 
	n +\ell-k \equiv 0 \pmod{p^e},
	\] 
	we have
	$(\widetilde{\alpha}\circ\sigma_2, \widetilde{\alpha}\circ\widetilde{\iota}_{h}\circ\sigma_2) \in \Stab_{B(G)}(t)$.

	On the other hand,  if $\sigma_i=\sigma_2$ and $\sigma_j=\sigma_1$ then we have
	\[
	n \equiv -k \pmod{p^e}, \ \ \ 
	\text{and} \ \ \ 
	k+m-\ell \equiv 0 \pmod{p^e},
	\] 
	Note that these congruences and the ones above do not have a common solution.
	Thus the result of the theorem follows.
\end{proof}

\begin{rmk}
\label{order of stabilizer in B(G)}
If $G$ is a $p$-group given in \cref{stabilizer in B(G)}, then $|\Stab_{B(G)}(t)|=|J(G)|^2$ or $3|J(G)|^2$ for a Beauville structure $t$ for $G$.
\end{rmk}

\vspace{5pt}

 The following lemma counts the number of Beauville structures for $G$ that satisfy the congruence conditions in \cref{stabilizers in B(G)-case distinction}.

\begin{lem}
	\label{number of B.S. according to exponents}
	Let $p\geq 5$ be a prime and let $G$ be a Beauville $p$-group given by the
	following presentation:
	\begin{equation*}
		G=\langle a,b \mid a^{p^e}=b^{p^e}=[b,a]^{p^j}=[b,a,b]=[b,a,a]=1 \rangle,
	\end{equation*}
	where $0<j\leq e$.
	Then the number of Beauville structures
	\[
	t=
	\Big( \big(x,y, y^{-1}x^{-1}\big), \big(x^ky^{\ell}c, x^my^nd,  (x^my^nd)^{-1}(x^ky^{\ell}c)^{-1}\big)\Big)
	\]
	satisfying the congruences
	$m\equiv -\ell \pmod{p^e}$ and $n +\ell-k \equiv 0 \pmod{p^e}$ is
	\begin{equation}
		\label{total number-first part}
		\begin{cases}
			(p+1)p(p-1)^3(p-2)|\Phi(G)|^3|G'|
			\ \ 
			&\text{if}
			\ \ 
			p\equiv -1 \pmod{3},
			 \\ \\
			(p+1)p(p-1)^3(p-4)|\Phi(G)|^3|G'| 
			\ \ 
			&\text{if}
			\ \ 
			p\equiv 1 \pmod{3}.
		\end{cases}
	\end{equation}
\end{lem}

\begin{proof}
Recall that by \cref{all B.S.-rmk}, any generating triple $\left(x,y,(xy)^{-1}\right)$ of $G$ appears in some Beauville structures for $G$. 
Thus the first triple $T_1=\big(x,y, y^{-1}x^{-1}\big)$ can be chosen arbitrarily. 
Hence there are $|\Phi(G)|^2 (p+1)p(p-1)^2$ choices for $T_1$.

On the other hand, for the second triple $T_2$, $c$ and $d$ can be chosen freely in $G'$, and  once we have $k$ and $\ell$, the exponents $m$ and $n$ are uniquely determined so that the congruences $m\equiv -\ell \pmod{p^e}$ and $n +\ell-k \equiv 0 \pmod{p^e}$ hold.

Since $x^ky^{\ell}c$ cannot fall in the maximal subgroups occupied by $T_1$, $p\nmid k$, $p\nmid \ell$ and $p \nmid k-\ell$. Thus, for the pair of $(k, \ell)$, we have  $p^{2e-2}(p-1)(p-2)$ possibilities. 
However, by \cref{restrictions on powers}, we have to remove the pairs of $(k,\ell)$ where 
\[\det \left(\begin{array}{cc} k & \ell\\ -\ell & k-\ell \end{array}\right)=k^2+\ell^2-k\ell \equiv 0 \pmod{p}.
\]
Note that if $p \equiv -1 \pmod{3}$ then the equation $k^2+\ell^2-k\ell  \equiv 0 \pmod{p}$ has no solution.
If $p\equiv 1 \pmod{3}$ then for any fixed $k$, there are $2$ distinct choices for $\ell$ in $\pmod{p}$ satisfying that equation (see page 200 in \cite{GJT}).
Hence there are $p^{2e-2}(p-1)(p-2)|G'|^2$ possibilities for the triple $T_2$ if $p \equiv -1 \pmod{p}$, and there are $p^{2e-2}(p-1)(p-4)|G'|^2$ possibilities otherwise.
Then the result follows from the fact that $|G'|=p^j$ and $|\Phi(G)|=p^{2e+j-2}$.
\end{proof}

\vspace{10pt}

\begin{proof}[Proof of Theorem A]
We will use Cauchy-Frobenius Lemma.
Let
\[
t=
\left( \big(x,y, y^{-1}x^{-1}\big), \big(x^ky^{\ell}c, x^my^nd,  (x^my^nc)^{-1}(x^ky^{\ell}d)^{-1}\big)\right)
\]
be a Beauville structure for $G$ where $1\leq k, \ell, m, n \leq p^e$ and $c, d \in G'$.

In \cref{stabilizer in B(G)}, we already determine $\Stab_{B(G)}(t)$.
Our next goal is to determine $\Stab_{\Au(G)}(t)$.

Observe first that if $t$ satisfies the first congruence conditions of \eqref{stabilizers in B(G)-case distinction}, then
$t'=(\sigma_0, \sigma_5)(t)$ is a Beauville structure satisfying the second congruence conditions of \eqref{stabilizers in B(G)-case distinction}.
This means that $t$ and $t'$ are in the same orbit, and hence the cardinality of their stabilizers are equal.
Thus, in determination of $\Stab_{\Au(G)}(t)$ we assume that we are  in either the first or third case of \eqref{stabilizers in B(G)-case distinction}.

By taking into account that $S\subseteq \Stab_{B(G)}(t)$, for simplicity we assume that $\big(\widetilde{\theta}\circ\widetilde{\iota}_g\circ\sigma_i, \widetilde{\theta}\circ\sigma_j\big)\circ \tau \in \Stab_{\Au(G)}(t)$. Then we have
\[
\left( \big(\widetilde{\theta}\circ\widetilde{\iota}_g\circ\sigma_i, \widetilde{\theta}\circ\sigma_j\big)\circ \tau\right)^2
=
\big( \widetilde{\theta} \circ \widetilde{\iota}_g \circ \widetilde{\theta}\circ \sigma_i\circ \sigma_j,  \widetilde{\theta}^2\circ \widetilde{\iota}_g \circ\sigma_j\circ \sigma_i \big) \in \Stab_{B(G)}(t).
\]

This, together with \cref{stabilizer in B(G)}, implies the following modulo $J(G)\times J(G)$:
\[
(\sigma_i\circ\sigma_j, \sigma_j\circ\sigma_i) \in \{ (\sigma_0, \sigma_0), (\sigma_1, \sigma_1), (\sigma_2, \sigma_2) \}.
 \]

One can observe that it is enough to analyze only the case $(\sigma_i\circ\sigma_j, \sigma_j\circ\sigma_i)=(\sigma_0, \sigma_0)$.
This means that modulo $J(G)$, we have $\sigma_j= \sigma_i^{-1}$.
So we have 6 possibilities:
\begin{equation}
	\label{six possibilities}
	(\sigma_i, \sigma_j) \in \{ (\sigma_0, \sigma_0), (\sigma_1, \sigma_2), (\sigma_2, \sigma_1), (\sigma_3, \sigma_3), (\sigma_4, \sigma_4), (\sigma_5, \sigma_5) \}.
\end{equation}

Let us start with $(\sigma_0, \sigma_0)$, that is, $(\widetilde{\theta}\circ\widetilde{\iota}_g, \widetilde{\theta})\circ\tau \in \Stab_{\Au(G)}(t)$.
Now we have
\[
\theta(x)=x^ky^{\ell}c \ \ \ \text{and} \ \ \ \theta(y)=x^my^nd.
\]
Then modulo $G'$ we have $\theta(x^ky^{\ell})=x$ and $\theta(x^my^n)=y$.
It then follows that
\begin{equation}
	\label{general case-exponents}
	n \equiv -k \pmod{p^e} \ \ \
	\text{and} \ \ \
	k^2+\ell m \equiv 1 \pmod{p^e},
\end{equation}

Note that $\theta([x,y])=[x,y]^{-1}$. 
We want $\theta(x^ky^{\ell}c)^{\theta(g)}=x$ and  $\theta(x^my^{-k}d)^{\theta(g)}=y$. 
The former equality implies that
\[
x^{\theta(g)} [x,y]^{-\ell^2km-k\ell\binom{k}{2}+km\binom{\ell}{2}}c^{k-1}d^{\ell}=x.
\]
Hence we have
\begin{equation}
	\label{relation1}
	[x,y]^{-{\ell}^2km-k\ell\binom{k}{2}+km\binom{\ell}{2}}[x, \theta(g)]c^{k-1}=d^{-\ell}.
\end{equation}
By the equality $\theta(x^my^{-k}d)^{\theta(g)}=y$, we get
\[
y^{\theta(g)} [x,y]^{m^2k\ell-k\ell\binom{m}{2}+km\binom{-k}{2}}c^md^{-k-1}=y.
\]
Thus,
\begin{equation}
	\label{relation2}
	[x,y]^{m^2k\ell-k\ell\binom{m}{2}+km\binom{-k}{2}}[y, \theta(g)]c^m=d^{k+1}.
\end{equation}
Since $(k+1)$st power of \eqref{relation1} is equal to $-\ell$th power of \eqref{relation2}, we get
\begin{equation}
	\label{relation3}
	[x^{k+1}y^{\ell}, \theta(g)]=[x,y]^{f(k,\ell,m)}.
\end{equation} 
As before,  we conclude that $G'= \{[x^{k+1}y^{\ell},g] \mid g\in G\}$.
Thus, there is $g \in G$ satisfying \eqref{relation3}, and it is unique modulo
\[
C_G(x^{k+1}y^{\ell})= \langle x^{k+1}y^{\ell} \rangle Z(G).
\]
As a consequence,  for any $c \in G'$, there correspond $|G'|$ elements $d' \in G'$, which are the elements in the coset $dG'$, such that 
$(\widetilde{\theta}\circ\widetilde{\iota}_g, \widetilde{\theta})\circ\tau \in \Stab_{\Au(G)}(t)$.

In order to determine the number of $t$ whose stabilizer contains $(\widetilde{\theta}\circ\widetilde{\iota}_g, \widetilde{\theta})\circ\tau$, we also need to analyze \eqref{general case-exponents}.
Observe that once we have $k$ and $\ell$, then $m$ and $n$ are unique.
Since all elements in $t$ fall in distinct maximal subgroups, 
this together with congruence $\ell m \equiv 1-k^2 \pmod{p^e}$ imply that  $k\not\equiv 0, 1, -1 \pmod{p}$.
Hence there are $p^{e-1}(p-3)$ possibilities  for $k$.
Once $k$ is chosen there are $p^{e-1}(p-5)$possibilities for $\ell$ since 
$\ell\not \equiv 0, k, \frac{k^2-1}{k}, k-1, k+1 \pmod{p}$ (see page 201 in \cite{GJT}).
Thus, the number of pairs of $(k ,\ell)$ is $p^{2e-2}(p-3)(p-5)$.
Hence, once the first triple in $t$ has been fixed, there are
\begin{align*}
	&p^{2e-2}(p-3)(p-5)|G'|^2\\
	=& (p-3)(p-5)|\Phi(G)|G'|
\end{align*}
choices for the second triple in $t$. Consequently, the total number of Beauville structures $t$ for $G$ such that $(\widetilde{\theta}\circ\widetilde{\iota}_g, \widetilde{\theta})\circ\tau \in \Stab_{\Au(G)}(t)$ is
 \begin{equation}
 	\label{six possibilities-total number}
 (p+1)p(p-1)^2(p-3)(p-5)|\Phi(G)|^3|G'|.
 \end{equation}
 Recall that the same arguments apply for any pair of $(\sigma_i, \sigma_j)$ in \eqref{six possibilities}.
 
 In order to determine the order of all stabilizers, it remains to find the number of Beauville structures in \eqref{six possibilities-total number} which belongs to the first case in  \eqref{stabilizers in B(G)-case distinction}.
 
 Assume that we are in the first case in \eqref{stabilizers in B(G)-case distinction}, that is, $\ell+m\equiv 0\pmod{p^e}$ and $n+\ell-k \equiv 0 \pmod{p^e}$.
 We start with $(\sigma_0, \sigma_0)$. 
 Then by above calculations we also know that \eqref{general case-exponents} holds.
 Thus, we have
 \begin{equation}
 	\label{distinguished-1}
 	\ell\equiv 2k \pmod{p^e} \ \
 	\text{and}
 	\ \
 	3k^2+1\equiv 0\pmod{p^e}.
 \end{equation}

If $(\sigma_i, \sigma_j)=(\sigma_2, \sigma_1)$, then one can check that
\begin{equation}
	\label{distinguished-2}
	k\equiv 2\ell \pmod{p^e} \ \
	\text{and}
	\ \
	3\ell^2+1\equiv 0\pmod{p^e},
\end{equation}
and if $(\sigma_i, \sigma_j)=(\sigma_1, \sigma_2)$, then 
\begin{equation}
	\label{distinguished-3}
	\ell\equiv -k \pmod{p^e} \ \
	\text{and}
	\ \
	3k^2+1\equiv 0\pmod{p^e}.
\end{equation}

Note that the congruence $3X^2+1 \equiv 0 \pmod {p^e}$ have $2$ distinct solutions if $p\equiv 1 \pmod{3}$, otherwise it has no solution (see page 201 in \cite{GJT}).
This means that if $p\equiv 1 \pmod{3}$, then there are
\begin{equation}
	\label{distinguished-total number}
	2|\Phi(G)|^2(p+1)p(p-1)^2|G'||G'|
\end{equation}
Beauville structures in each of cases \eqref{distinguished-1}-\eqref{distinguished-3}.

We continue with $(\sigma_4, \sigma_4)$.
Since we are in the first case of \eqref{stabilizers in B(G)-case distinction}, if
$(\widetilde{\theta}\circ \widetilde{\iota}_g\circ\sigma_4 , \widetilde{\theta}\circ\sigma_4)\circ\tau \in \Stab_{\Au(G)}(t)$, then both
$(\widetilde{\theta_1}\circ \widetilde{\iota}_{g_1}\circ\sigma_3 , \widetilde{\theta_1}\circ\sigma_3)\circ\tau$ and
$(\widetilde{\theta_2}\circ \widetilde{\iota}_{g_2}\circ\sigma_5 , \widetilde{\theta_2}\circ\sigma_5)\circ\tau$ are in $\Stab_{\Au(G)}(t)$ for some $g_1, g_2\in G$ and $\theta_1, \theta_2\in \Aut(G)$.

The condition $(\widetilde{\theta}\circ \widetilde{\iota}_g\circ\sigma_4 , \widetilde{\theta}\circ\sigma_4)\circ\tau \in \Stab_{\Au(G)}(t)$ implies that
\begin{equation}
	\label{distinguished-4}
	k^2+\ell^2-k\ell\equiv 1 \pmod{p^e}.
\end{equation}
Let us denote by $K$ the number of pair of $(k ,\ell)$ satisfying \eqref{distinguished-4} so that $t=(T_1, T_2)$ form a Beauville structure for $G$. 
Hence the number of $t$ in the first case of \eqref{stabilizers in B(G)-case distinction} such that $(\widetilde{\theta}\circ \widetilde{\iota}_g\circ\sigma_4 , \widetilde{\theta}\circ\sigma_4)\circ\tau \in \Stab_{\Au(G)}(t)$
is
\begin{equation}
	\label{distinguished-total number-2}
	K|\Phi(G)|^2(p+1)p(p-1)^2|G'||G'|.
\end{equation}

Notice that a pair $(k, \ell)$ that satisfies any one of \eqref{distinguished-1}--\eqref{distinguished-3} or \eqref{distinguished-4} does not satisfy the others.
This implies that any Beauville structure satisfying one of  \eqref{distinguished-1}--\eqref{distinguished-3} or \eqref{distinguished-4} has stabilizer of order $6|J(G)|^2$.

Putting everything together, we can finally calculate the number of orbits in $\U(G)$ under the action of $\Au(G)$ by using Cauchy-Frobenius Lemma.

We first assume that $p\equiv 1 \pmod{3}$.
Let us denote by $A$ the number in \eqref{total number-first part} when $p\equiv 1 \pmod{3}$.
Then total number of Beauville structures in the first two cases of \eqref{stabilizers in B(G)-case distinction} is $2A$.
Let us denote by $B$ the number in \eqref{distinguished-total number-2}, by $C$ the number in \eqref{six possibilities-total number} and by $D$ the number in \eqref{distinguished-total number}.
Then out of $2A$ Beauville structures in the first two cases of \eqref{stabilizers in B(G)-case distinction}, $2B+6D$ of them have stabilizers of order $6|J(G)|^2$ and hence $2A-2B-6D$ of them have stabilizers of order $3|J(G)|^2$.

On the other hand, the total number of Beauville structures in the last case of \eqref{stabilizers in B(G)-case distinction} is $|\U(G)|-2A$. Out of these Beauville structures, $6C-6B-6D$ of them have stabilizers of order $2|J(G)|^2$, and the remaining $|\U(G)|-2A-6C+6B+6D$ of them have stabilizers of order $|J(G)|$.

Thus,
\[
\sum_{t \in \U(G)}|\Stab_{\Au(G)}(t)|=\left(|\U(G)|+4A+6C+12D\right)|J(G)|^2.
\]

Now recall that by \cref{order of the group}, $|\Au(G)|=72 |J(G)|^2 .|\Inn(G)|. |\Aut(G)|$. Then by Cauchy-Frobenius Lemma, the total number of distinct orbits is
\begin{equation}
	\label{orbits-split extension}
	\begin{aligned}[b]
		&\frac{1}{|\Au(G)|}\sum_{t \in \U(G)}|\Stab_{\Au(G)}(t)|=\\
		&\begin{aligned}[b]
			\frac{1}{72|\Inn(G)||\Aut(G)|}
			\left( 
			|\U(G)|
			+4A+6C+12D
			\right),
		\end{aligned} \\
	\end{aligned}
\end{equation}
where
\[
|\U(G)|=|\Phi(G)|^4(p+1)p(p-1)^3(p-2)(p-3)(p-4),
\] 
$|\Inn(G)|=|G'|^2$ \ and  \ $|\Aut(G)|=(p+1)p(p-1)^2 |\Phi(G)|^2$,
by \cref{number of all B.S.}, \cref{class2-case3-center} and \cref{automorphism-class 2-case 3}, respectively.

Consequently, as $|G'|=p^j$ and $|\Phi(G)|=p^{2e+j-2}$, we have that if $p\equiv 1 \pmod{3}$, \cref{orbits-split extension} is equal to
\[
\frac{1}{72}
\bigg( p^{4e-4}
(p-1)(p-2)(p-3)(p-4)
+ \
p^{2e-2}
\Big(4(p-1)(p-4)+6(p-3)(p-5)\Big)
+
24
\bigg).
\]

By applying similar calculations, it is easy to check that, if $p\equiv -1 \pmod{3}$, this number is equal to
\[
\frac{1}{72}\bigg(p^{4e-4}
(p-1)(p-2)(p-3)(p-4)
+ \
p^{2e-2}  
\Big(4(p-1)(p-2)+6(p-3)(p-5)\Big)  
\bigg).
\]
\end{proof}

\begin{rmk}
	\label{rmk for abelian}
One can check that \cref{stabilizer in B(G)}, \cref{number of B.S. according to exponents} and Theorem A also hold when $j=0$ in the presentation of $G$, that is, when $G\cong C_{p^e}\times C_{p^e}$.
As a consequence, for $G\cong C_{p^e}\times C_{p^e}$ we get the same result as in Corollary 4 in \cite{GJT}.
\end{rmk}

We next continue with non-abelian metacyclic Beauville $p$-groups.

\vspace{10pt}

\begin{proof}[Proof of Theorem C]
Let $t=(T_1, T_2)$ be a Beauville structure for $G$. Suppose that
\[
(\widetilde{\alpha}\circ\beta_w \circ\sigma_i, \
\widetilde{\alpha}\circ\widetilde{\iota}_g\circ\beta_{w'}\circ\sigma_j ) \in \Stab_{\Au(G)}(t)
\]
for some for some $\alpha \in \Aut(G)$, $g\in G$.

Recall that by the proof of \cref{number of all B.S.}, every Beauville structure for $G$ is inherited by $G/\Phi(G)\cong C_p\times C_p$.
Then by \cref{stabilizer in B(G)} and \cref{rmk for abelian}, we have that  $\sigma_i=\sigma_0, \sigma_1$ or $\sigma_2$.

Assume that $\sigma_i=\sigma_2$. 
Let $T_1=(x, y, y^{-1}x^{-1})$.
Then we have \[
\alpha \left(x^{w(y^{-1}x^{-1}, x)}\right)=y\]
and 
\[\alpha\left(y^{w(y^{-1}x^{-1}, x)}\right)= y^{-1}x^{-1}.\] 
If we write $x= a^ib^jg^p$ and $y=a^kb{\ell}h^p$ for some $0\leq i, j,  k, \ell \leq p-1$, then
\[
\alpha(x)\equiv a^kb^{\ell} \pmod{\Phi(G)} \ \
\text{and} \ \
\alpha(y)\equiv a^{-i-k}b^{-j-\ell} \pmod{\Phi(G)}. 
\]
On the other hand, by \cref{automorphism-powerful}, we have
\[
\alpha(x)\equiv a^{ni+sj}b^j \pmod{\Phi(G)} \ \
\text{and} \ \
\alpha(y)\equiv a^{nk+s\ell}b^{\ell} \pmod{\Phi(G)}.
\]
Since $p\geq 5$, it follows that $p \mid \ell$ and $p \mid j$, and hence
$x, y \in \langle a, \Phi(G)\rangle$., which is not possible.
The same argument applies if $\sigma_i=\sigma_1$. 
As a consequence,  we have $\sigma_i=\sigma_j=\sigma_0$.
 
Now let  $T_2=(u, v, v^{-1}u^{-1})$. 
Since $\sigma_i=\sigma_j=\sigma_0$, by applying similar arguments in the proof of \cref{stabilizer in B(G)}, we have
\[
\Stab_{B(G)}(t)
=
\big\{ 
\big(  \widetilde{\iota}_{w(x,y)^{-1}}\circ \beta_w, \ \widetilde{\iota}_{w'(u,v)^{-1}}\circ \beta_{w'}  \big)
\mid
\beta_w, \beta_{w'} \in J(G)
\big \},
\]
for any Beauville structure $t$,
and as a consequence $|\Stab_{B(G)}(t)|=|J(G)|^2$.

We next deal with $\Stab_{\Au(G)}(t)$. We have two possibilities: either we have $\Stab_{\Au(G)}(t)=\Stab_{B(G)}(t)$ and hence $|\Stab_{\Au(G)}(t)|= |J(G)|^2$, or we have
$|\Stab_{\Au(G)}(t)|= 2|J(G)|^2$. 

Let us determine the number of $t \in \U(G)$ such that $|\Stab_{\Au(G)}(t)|=2|J(G)|^2$.
Assume that $(\widetilde{\alpha}\circ \mu, \widetilde{\alpha}\circ\epsilon)\circ \tau \in \Stab_{\Au(G)}(t)$ for some $\mu, \epsilon \in I(G)$.
Then
\[
\big((\widetilde{\alpha}\circ \mu, \widetilde{\alpha}\circ\epsilon)\circ \tau\big)^2=
(\widetilde{\alpha}\circ\mu \circ \widetilde{\alpha} \circ \epsilon, \widetilde{\alpha}\circ\epsilon \circ\widetilde{\alpha} \circ \mu) \in \Stab_{B(G)}(t).
\]
This implies that $\widetilde{\alpha}\circ\epsilon\circ \widetilde{\alpha}  \in I(G)$ which is true if and only if $\widetilde{\alpha}^2 \in I(G)$.
Then by \cref{intersections-metacyclic}, we have $\alpha^2 \in \Inn(G)$.
On the other hand,  since elements in $T_1 \cup T_2$ fall into distinct maximal subgroups and $\epsilon$ `preserves' the maximal branches of $G$, the condition $(\widetilde{\alpha}\circ \mu, \widetilde{\alpha}\circ\epsilon)\circ \tau \in \Stab_{\Au(G)}(t)$ implies that $\alpha \notin \Inn(G)$.
Consequently, $\alpha$ is an involution in $\Out(G)$.

Let us fix an involution $\alpha$ in $\Out(G)$. Once we fix $T_1$, for any $\epsilon \in I(G)$ we take 
$\widetilde{\alpha} \circ \epsilon(T_1)=T_2$.
Now the question is: for such a fixed $\alpha \in \Aut(G)$ and $T_1$, how many possibilities do we have for $T_2$? 
This is equivalent to determine the number of distinct $(T_1)\epsilon$ for $\epsilon \in I(G)$. 
Since every $\epsilon \in I(G)$ can be written of the form $\widetilde{\iota}_h \circ \beta_w \circ \sigma_i$ and for $T_1$ the action of $\widetilde{\iota}_h \circ \beta_w$ is a conjugation, we have $6|\Inn(G)|$ possibilities for $T_2$.

We next determine the number of such $T_1=(x, y, y^{-1}x^{-1})$. 
Note that any $\epsilon \in I(G)$ `preserves' the maximal branches of $T_1$.
On the other hand, by \cref{powerful-elements of Out(G)}, we know that $\alpha$ is of the form
\begin{align*} 
	\alpha \colon a & \longmapsto a^{-1}w_a \\
	b & \longmapsto a^{\lambda}bw_b,
\end{align*}
for some $w_a, w_b\in \Phi(G)$ and $0\leq \lambda\leq p-1$. 
Observe that there are two maximal subgroups which are invariant under $\alpha$: 
$\langle a, \Phi(G) \rangle$ and $\langle ab^{2\lambda^{-1}}, \Phi(G) \rangle$ if $\lambda\neq 0$, or $\langle a, \Phi(G) \rangle$ and $\langle b, \Phi(G) \rangle$ if $\lambda=0$.
Since all elements in $T_1\cup T_2$ lie in different maximal subgroups of $G$ and we want $\widetilde{\alpha}\circ \epsilon(T_1)=T_2$, the elements of $T_1$ cannot be in the maximal subgroups which are invariant under $\alpha$.
Thus, we have $(p-1)(p-1)|\Phi(G)|$ possibilities for $x$. 

Once we have chosen $x$, the maximal subgroup containing $y$ cannot contain $x$ and $\alpha(x)$.
Thus we can choose that maximal subgroup containing $y$ in $(p-3)$ different ways. 
Next we have to replace $y$ with $y^j$ for some $1\leq j \leq p-1$ so that the element $xy^j$ is neither in the invariant maximal subgroups under $\alpha$ nor in the maximal subgroups containing $\alpha(x)$ and $\alpha(y^j)$. 
Hence $y$ can be chosen $(p-3)(p-5)|\Phi(G)|$ different ways, and thus for the triple $T_1$ we have
\[
(p-1)^2(p-3)(p-5)|\Phi(G)|^2
\]
different possibilities.

Consequently, for a fixed such $\alpha \in \Aut(G)$ the number of Beauville structures of $G$ which have a stabilizer of order $2|J(G)|^2$ is
\[
6(p-1)^2(p-3)(p-5)|\Phi(G)|^2|\Inn(G)|.
\]
Since by \cref{powerful-elements of Out(G)}, the number of involutions in $\Out(G)$ is $p^{2i}$, 
the number of Beauville structures for $G$ which have a stabilizer of order $2|J(G)|^2$ is
\[
6p^{2i}(p-1)^2(p-3)(p-5)|\Phi(G)|^2|\Inn(G)|
=
6 |\Phi(G)|^2 |G|(p-1)^2(p-3)(p-5).
\]

We are now ready to complete the proof. 
Now recall that by \cref{order of the group}, $|\Au(G)|=72 |J(G)|^2 .|\Inn(G)|. |\Aut(G)|$, where $|\Aut(G)|=(p-1)p^{2e+2i-1}$ by \cref{automorphism-powerful}.
Then by Cauchy-Frobenius Lemma, the total number of distinct orbits is
\begin{align*}
	&\frac{1}{|\Au(G)|}\sum_{t \in \U(G)}|\Stab_{\Au(G)}(t)|=\\
	&\frac{1}{72|\Inn(G)|. |\Aut(G)|}
	\Big( |\U(G)|+ 6 |\Phi(G)|^2 |G|(p-1)^2(p-3)(p-5)\Big)\
\end{align*}
where  \[
|\U(G)|=|\Phi(G)|^4(p+1)p(p-1)^3(p-2)(p-3)(p-4),
\]
 by \cref{number of all B.S.}.
 Thus, the number of isomorphism classes of Beauville surfaces with $G$ is
 \[
 \frac{1}{72}
 \Big(
 p^{4e-6} (p+1)(p-1)^2(p-2)(p-3)(p-4)
 +
 6p^{2e-3}(p-1)(p-3)(p-5)
 \Big).
 \]	
\end{proof}

We close the paper by giving the proof of Theorem B.

\vspace{10pt}

\begin{proof}[Proof of Theorem B]
	We are going to apply the same arguments in the proof of Theorem C.
	Then by taking into account \cref{class2-case2-elements of Out(G)}, the number of Beauville structures for $G$ which have stabilizers of order $2|J(G)|^2$ is
	\[
	6p^{2i-j}(p-1)^2(p-3)(p-5)|\Phi(G)|^2|\Inn(G)|.
	\]
	Now by \cref{class2-case2-inner}, this number is equal to
	\[
	6p^{2e+4i+3j-4}(p-1)^2(p-3)(p-5).
	\]
	The rest of the Beauville structures have stabilizers of order $|J(G)|^2$.
	Finally if we apply the Cauchy-Frobenius Lemma by using \cref{number of all B.S.} and \cref{order of the group}, we obtain that the total number of distinct orbits is
\[
	\frac{1}{72}
		\left(
		p^{4e-6}(p+1)(p-1)^2(p-2)(p-3)(p-4)
		+
		6p^{2e-j-3}(p-1)(p-3)(p-5)
		\right).
\]
\end{proof}

\section{Acknowledgments}
I would like to thank G. A. Fern\'{a}ndez-Alcober for his valuable feedback and
helpful suggestions. I also thank the University of the Basque Country
for its hospitality while this research was conducted.

\end{document}